\RequirePackage{fix-cm}
\documentclass[smallextended]{svjour3}       
\smartqed  

\usepackage{array}
\usepackage{caption,color}
\usepackage{subfigure}
\usepackage{float}
\usepackage{textcomp}
\usepackage{amsmath}
\usepackage{stfloats}
\usepackage{url}
\usepackage{verbatim}
\usepackage{amssymb}
\usepackage{bm,amsfonts}
\usepackage{longtable}
\usepackage{textcomp}
\usepackage{braket}
\usepackage{multirow}
\usepackage{cases}
\usepackage{graphicx}

\usepackage{tabularx}
\usepackage{algorithm}
\usepackage{algorithmic}
\usepackage{epstopdf}
\usepackage{setspace}
\usepackage[caption=false]{subfig}
\newcommand{\ii}{{\bf{i}}}
\newcommand{\jj}{{\bf{j}}}
\newcommand{\kk}{{\bf{k}}}
\begin{document}

\title{Improved power methods for computing eigenvalues of dual quaternion Hermitian matrices\thanks{This work was supported by the National Natural Science Foundation of China (12171271).}
}



\author{Yongjun Chen  \and Liping~Zhang}


\institute{Yongjun Chen\at
            Department of Mathematical Sciences, Tsinghua University, Beijing 100084, China\\
              yongjun-24@mails.tsinghua.edu.cn
           \and
            Liping Zhang, Corresponding author\at
            Department of Mathematical Sciences, Tsinghua University, Beijing 100084, China\\
            lipingzhang@tsinghua.edu.cn
}

\date{Received: date / Accepted: date}

\maketitle

\begin{abstract}
This paper investigates the eigenvalue computation problem of the dual quaternion Hermitian matrix closely related to multi-agent group control. Recently, power method was proposed by Cui and Qi in {\it Journal of Scientific Computing, 100 (2024)} to solve such problem. Recognizing that the convergence rate of power method is slow due to its dependence on the eigenvalue distribution, we propose two improved versions of power method based on dual complex adjoint matrices and Aitken extrapolation, named  DCAM-PM and ADCAM-PM. They achieve notable efficiency improvements and demonstrate significantly faster convergence. However, power method may be invalid for dual quaternion Hermitian matrices with eigenvalues having identical standard parts but distinct dual parts. To overcome this disadvantage, utilizing the eigen-decomposition properties of dual complex adjoint matrix, we propose a novel algorithm EDDCAM-EA which surpasses the power method in both accuracy and speed.
Application to eigenvalue computations of dual quaternion Hermitian matrices in multi-agent formation control and numerical experiments highlight the remarkable accuracy and speed of our proposed algorithms.

\end{abstract}

\section{Introduction}
Dual quaternion matrices play a pivotal role in robotics research, addressing fundamental problems such as hand-eye calibration \cite{Daniilidis1999}, simultaneous localization and mapping  \cite{qc1,c1,Cadena2016,Grisetti2010}, and kinematic modeling \cite{Qi2023}. Among these, dual quaternion Hermitian matrices, a crucial subclass, exhibit remarkable properties essential for analyzing multi-agent formation control system \cite{Qi2023} and solving pose graph optimization problem \cite{Cadena2016,Cui2023,Grisetti2010}.

Recent advancements have significantly expanded the spectral theory of dual quaternion matrices \cite{Duan2023,d2,d3,Li2023,wang2023dual,wei2024singular}. Qi and Luo \cite{Li2023} established that an $n$-by-$n$ dual-quaternion Hermitian matrix has exactly $n$ eigenvalues, all of which are dual numbers, and provided the eigenvalue decomposition for such matrices. Ling et al. \cite{d3} extended the minimax principle and Fan-Hoffman inequality to dual quaternion matrices. Moreover, von Neumann-type trace inequalities and Hoffman-Wielandt-type inequalities for dual quaternion Hermitian matrices were developed in \cite{ling2022neumann}.

To compute strict dominant eigenvalues of dual quaternion Hermitian matrices, Cui and Qi \cite{Cui2023} introduced a power method with linear convergence. A strict dominant eigenvalue, characterized by its standard part exceeding that of other eigenvalues, was efficiently computed using this method. However, its computational cost escalates with matrix dimension, and the method becomes ineffective when eigenvalues have identical standard parts but differ in their dual parts. These limitations underscore the need for more efficient and robust algorithms for eigenvalue computation in dual quaternion Hermitian matrices.

Inspired by the role of complex adjoint matrices in quaternion matrix theory \cite{Zhang1997}, Chen et al. \cite{chen2024dual} proposed the dual complex adjoint matrix, a novel construct tailored for dual quaternion matrices. They established a ring isomorphism between the dual quaternion matrix ring and the dual complex adjoint matrix ring, enabling the definition and uniqueness proof of representative right eigenvalues. Leveraging these properties, a direct solution to the hand-eye calibration problem and a Rayleigh quotient iteration method were introduced, demonstrating the versatility of dual complex adjoint matrices.

This paper aims to deepen the understanding of the relationship between dual complex adjoint matrices and dual quaternion matrices in the context of eigenvalue problems. Furthermore, we seek to develop more efficient algorithms for computing the eigenvalues of dual quaternion Hermitian matrices, addressing the limitations of existing methods.

The main contributions of this paper are summarized as follows.
\begin{itemize}
\item[(i)] We establish a crucial connection between the eigenvalue problems of dual quaternion matrices and dual complex adjoint matrices. This transformation simplifies the computation of eigenpairs for dual quaternion matrices into an equivalent problem for dual complex matrices. Building on this insight, we refine the power method and introduce a novel variant, DCAM-PM, which significantly enhances computational efficiency.
This improvement is confirmed in Table \ref{table1}.

\item[(ii)]  Since the convergence rate of the power method is influenced by the distribution of eigenvalues, we accelerate its convergence by integrating Aitken extrapolation into DCAM-PM, resulting in an enhanced algorithm, ADCAM-PM. We demonstrate that ADCAM-PM achieves a faster convergence rate compared to the power method. This improvement is validated in Figure \ref{fig:testfig} and Table \ref{table0}.

\item[(iii)] 

To overcome the limitations of the power method, which fails when solving for eigenvalues of dual quaternion Hermitian matrices with eigenvalues sharing identical standard parts but differing dual parts, we fully leverage the properties of dual complex adjoint matrices and propose EDDCAM-EA, a method capable of comprehensively computing all eigenpairs of dual quaternion Hermitian matrices. This method significantly outperforms the power method (including our improved variants, DCAM-PM and ADCAM-PM,) in terms of both accuracy and speed, while effectively addressing its limitations. Validation is provided in Table \ref{table1} and Subsection \ref{subsection6.3}.
\end{itemize}


This paper is organized as follows. Section \ref{Preliminaries} provides an overview of key results related to dual quaternion matrices. Section \ref{section3} establishes the relationship between dual complex adjoint matrices and dual quaternion matrices in the context of eigenvalue problems. Section \ref{section5} introduces DCAM-PM, an improved power method based on the dual complex adjoint matrix, and further refines it via Aitken extrapolation to derive ADCAM-PM, an accelerated variant with superior convergence. Section \ref{section4} presents EDDCAM-EA, a novel approach for computing all eigenpairs of dual quaternion Hermitian matrices through eigen-decomposition of dual complex adjoint matrices. Section \ref{section6} presents numerical experiments that address the eigenvalue computation problem of dual quaternion Hermitian matrices, including its application in Multi-Agent Formation Control, demonstrating the efficacy of the proposed methods. Finally, concluding remarks are provided in Section \ref{section7}.

\section{Preliminary}\label{Preliminaries}

\subsection{Dual complex numbers and dual quaternions}

We denote $\mathbb{R}$ and $\mathbb{C}$ as sets of real numbers and complex numbers, respectively. Denote the set of dual complex numbers as
\begin{displaymath}
   \mathbb{DC}=\{{a}=a_{st}+a_{\mathcal I}\varepsilon|a_{st},a_{\mathcal I}\in \mathbb{C}\},
  \end{displaymath}
where $\varepsilon$ is an infinitesimal unit that satisfies $\varepsilon\neq 0$ and $\varepsilon^{2}=0$.  We call $a_{st}$ and $a_{\mathcal I}$ the standard and dual parts of $a$, respectively. The infinitesimal unit $\varepsilon$ is commutative in the multiplication with complex numbers and quaternions. We state that $a$ is {\it appreciable} if $a_{st}\neq 0$. The {\it conjugate} \cite{Qi2022} of $a$ is
${a}^{\ast }=a_{st}^{\ast}+a_{\mathcal I}^{\ast }\varepsilon,$ where $a_{st}^{\ast}$ and $a_{\mathcal I}^{\ast}$ are the conjugates of $a_{st}$ and $a_{\mathcal I}$, respectively. The {\it absolute value} \cite{Qi2022} of ${a}$ is
 \begin{displaymath}\left |{a}\right |=\begin{cases}
\left | a_{st}\right |+{\rm sgn}( a_{st}) a_{\mathcal I}\varepsilon, & {\rm if} \quad a_{st}\neq 0, \\
\left | a_{\mathcal I} \right |\varepsilon, & \text{\rm otherwise.}
\end{cases}
\end{displaymath}
If $a_{st},a_{\mathcal I}\in \mathbb{R}$, ${a}$ is called a {\it dual number}. Denote the set of dual numbers as
\begin{displaymath}
   \mathbb{D}=\{{a}=a_{st}+a_{\mathcal I}\varepsilon|a_{st},a_{\mathcal I}\in \mathbb{R}\}.
  \end{displaymath}
The order operation for dual numbers was defined in \cite{Cui2023}. For two dual numbers ${a}=a_{st}+a_{\mathcal I}\varepsilon$ and ${b}=b_{st}+b_{\mathcal I}\varepsilon$, we say that ${a} > {b}$ if
$$\text{$a_{st} > b_{st}$ ~~~{\rm or}~~~ $a_{st}=b_{st}$ {\rm and} $a_{\mathcal I} > b_{\mathcal I}$}.$$
When $b_{st}\ne 0$ or $a_{st}=b_{st}=0$ and $b_{\mathcal I}\ne 0$, we can define the division operation \cite{Qi2022} of dual numbers as
$$\frac{{a}}{{b}}=\begin{cases}
\dfrac{ a_{st}}{b_{st}}+\left (\dfrac{ a_{\mathcal I}}{b_{st}}-\dfrac{ a_{st}b_{\mathcal I}}{b_{st}b_{st}} \right )\varepsilon, & {\rm if} \quad b_{st}\neq 0, \\
\dfrac{ a_{\mathcal I}}{b_{\mathcal I}}+c\varepsilon, & {\rm if} \quad a_{st}=b_{st}=0,~b_{\mathcal I}\ne 0
\end{cases}
$$
where $c$ denotes an arbitrary complex number.

For two dual complex numbers ${a}=a_{st}+a_{\mathcal I}\varepsilon, {b}=b_{st}+b_{\mathcal I}\varepsilon\in \mathbb{DC}$, addition and multiplication of $a$ and $b$ are defined as
  \begin{displaymath}
  \begin{aligned}
  &{a}+{b}={b}+{a}=\left (a_{st}+b_{st} \right )+\left (a_{\mathcal I}+b_{\mathcal I} \right )\varepsilon,\\
&{a}{b}={b}{a}=a_{st}b_{st} +\left (a_{st}b_{\mathcal I}+a_{\mathcal I}b_{st} \right )\varepsilon.
\end{aligned}
\end{displaymath}

The limit operation for dual number sequences was defined in \cite{Qi2022}.
For a dual number sequence $\left \{a_{k}=a_{k,st}+a_{k,\mathcal I}\varepsilon | k=1,2,\ldots \right \}$, we say that this sequence converges to a limit $a=a_{st}+a_{\mathcal I}\varepsilon$ if
\begin{equation*}
\underset{k\rightarrow \infty }{\lim} a_{k,st}=a_{st}\quad {\rm and}\quad \underset{k\rightarrow \infty }{\lim} a_{k,\mathcal I}=a_{\mathcal I}.
\end{equation*}

Similar to the equivalent infinity symbol $O$ of real numbers, notation for the equivalent infinity of dual numbers was introduced in \cite{Cui2023}.
For a dual number sequence $\left \{a_{k}=a_{k,st}+a_{k,\mathcal I}\varepsilon | k=1,2,\ldots \right \}$ and a real number sequence $\{b_k|k=1,2,\ldots\}\subset\mathbb{R}$, we denote $a_k=O_D(b_k)$ if
$a_{k,st}=O(b_k)$ and $a_{k,\mathcal{I}}=O(b_k)$. Furthermore, if there exist a constant $0<c<1$ and a polynomial $h(k)$, such that
$a_k=O_D(c^kh(k))$, then we write $a_k=\tilde{O}_D(c^k)$.

Denote the set of quaternions as
\begin{displaymath}
 \mathbb{Q}=\{\tilde{q}=q_{0}+q_{1} \ii+q_{2} \jj+q_{3} \kk|~q_{0},q_{1},q_{2},q_{3}\in\mathbb{R}\},
  \end{displaymath}
 where $\ii, \jj, \kk$ are three imaginary units of quaternions, satisfying
 \begin{displaymath}
  \begin{aligned}
  &\ii\jj=-\jj\ii=\kk,\quad \jj\kk=-\kk\jj=\ii,\quad \kk\ii=-\ii\kk=\jj,\\
  &\ii^2=\jj^2= \kk^2=-1.
 \end{aligned}
\end{displaymath}
A quaternion $\tilde{q}\in\mathbb{Q}$ can also be represented as $\tilde{q}=\left [q_{0}, q_{1},q_{2},q_{3}\right ]$, which is a real four-dimensional vector. We can rewrite $\tilde{q}=[ q_{0},\vec{q} ]$, where $\vec{q}$ is a real three-dimensional vector \cite{Cheng2016,Daniilidis1999,Wei2013}. The conjugate of $\tilde{p}$ is $\tilde{p}^{\ast}=\left [p_{0}, -\vec{p}\right ]$ and the magnitude of $\tilde{p}$ is
$\left |\tilde{p}\right|=\sqrt{p_{0}^{2}+\left \| \vec{p} \right \|_2^2}$.

The zero element in $\mathbb{Q}$ is $\tilde{0}=\left [0,0,0,0\right ]$ and the unit element is $\tilde{1}=\left [1,0,0,0\right ]$. For any two quaternions $\tilde{p}=\left [p_{0}, \vec{p}\right ]$ and $\tilde{q}=\left [q_{0}, \vec{q}\right ]$, the addition and multiplicity of $\tilde{q}$ and $\tilde{p}$ are defined as
 \begin{displaymath}
\tilde{p}+\tilde{q}=\tilde{q}+\tilde{p}=\left [p_{0}+q_{0},\vec{p}+\vec{q}\right ],\quad
\tilde{p}\tilde{q}=\left [p_{0}q_{0}-\vec{p}\cdot \vec{q}, p_{0}\vec{q}+q_{0}\vec{p}+\vec{p}\times \vec{q}\right ].
\end{displaymath}
The multiplication of quaternions satisfies the distributive law, but is noncommutative. If $\tilde{q}\tilde{p}=\tilde{p}\tilde{q}=\tilde{1}$, the quaternion $\tilde{p}$ is said to be invertible and its inverse is $\tilde{p}^{-1}=\tilde{q}$.

We denote the set of unit quaternions by $\mathbb{U}=\{\tilde{p}\in \mathbb{Q}|~|\tilde{p} |=1\}$. For any two unit quaternion $\tilde{p},\tilde{q}\in \mathbb{U}$, we have $\tilde{p}\tilde{q}\in \mathbb{U}$ and $\tilde{p}^{\ast }\tilde{p}=\tilde{p}\tilde{p}^{\ast }=\tilde{1}$; that is, $\tilde{p}$ is invertible, and $\tilde{p}^{-1}=\tilde{p}^{\ast }$. In general, if a quaternion $\tilde{p}\ne \tilde{0}$, then we have  $\frac{\tilde{p}^{\ast }}{|\tilde{p}|^2}\tilde{p}=\tilde{p}\frac{\tilde{p}^{\ast }}{|\tilde{p}|^2}=\tilde{1}$; hence, $\tilde{p}^{-1}=\frac{\tilde{p}^{\ast }}{|\tilde{p}|^2}$.

Denote the set of dual quaternions as
 \begin{displaymath}
 \hat{\mathbb{Q}}=\{\hat{p}=\tilde{p}_{st}+\tilde{p}_{\mathcal I}\varepsilon|~\tilde{p}_{st}, \tilde{p}_{\mathcal I} \in \mathbb{Q}\}.
\end{displaymath}
We call $\tilde{p}_{st}$ and $\tilde{p}_{\mathcal I}$ the standard part and dual part of $\hat{p}$, respectively. If $\tilde{p}_{st}\neq \tilde{0}$, then we say that $\hat{p}$ is {\it appreciable}.
The conjugate of $\hat{p}$ is
$\hat{p}^{\ast }=\tilde{p}_{st}^{\ast}+\tilde{p}_{\mathcal I}^{\ast }\varepsilon$ and the magnitude \cite{Qi2022} of $\hat{p}$ is
\begin{equation*}
\left | \hat{p}\right |=\begin{cases}
\left | \tilde{p}_{st}\right |+\frac{{\rm sc}(\tilde{p}_{st}^{\ast }\tilde{p}_{\mathcal I})}{\left |\tilde{p}_{st} \right |}\varepsilon, & \text{if} \quad \tilde{p}_{st}\neq \tilde{0}, \\
\left | \tilde{p}_{\mathcal I} \right |\varepsilon, & \text{if} \quad \tilde{p}_{st}= \tilde{0},
\end{cases}
\end{equation*}
where ${\rm sc}(\tilde{p})=\frac{1}{2}(\tilde{p}+\tilde{p}^{\ast})$.
We say that $\hat{p}$ is a {\it unit dual quaternion } if $\left |\hat{p} \right |=1$.

The zero element in $\hat{\mathbb{Q}}$ is $\hat{0}=\tilde{0}+\tilde{0}\varepsilon$ and the unit element is $\hat{1}=\tilde{1}+\tilde{0}\varepsilon$. For two dual quaternions $\hat{p}=\tilde{p}_{st}+\tilde{p}_{\mathcal I}\varepsilon, \hat{q}=\tilde{q}_{st}+\tilde{q}_{\mathcal I}\varepsilon$, addition and multiplicity of $\hat{p}$ and $\hat{q}$ are
\begin{equation*}
\hat{p}+\hat{q}=\hat{q}+\hat{p}=\left (\tilde{p}_{st}+\tilde{q}_{st} \right )+\left (\tilde{p}_{\mathcal I}+\tilde{q}_{\mathcal I} \right )\varepsilon,\qquad
\hat{p}\hat{q}=\tilde{p}_{st}\tilde{q}_{st} +\left (\tilde{p}_{st}\tilde{q}_{\mathcal I}+\tilde{p}_{\mathcal I}\tilde{q}_{st} \right )\varepsilon.
\end{equation*}
If $\tilde{q}_{st}\ne \tilde{0}$ or  $\tilde{p}_{st}=\tilde{q}_{st}=\tilde{0}$ and $\tilde{q}_{\mathcal I}\ne\tilde{0}$, the division \cite{Cui2023} of $\hat{p}$ and $\hat{q}$ is defined as
\begin{equation*}
\frac{ \tilde{p}_{st}+ \tilde{p}_{\mathcal I}\varepsilon}{\tilde{q}_{st}+ \tilde{q}_{\mathcal I}\varepsilon}=\begin{cases}
\dfrac{ \tilde{p}_{st}}{\tilde{q}_{st}}+\left (\dfrac{ \tilde{p}_{\mathcal I}}{\tilde{q}_{st}}-\dfrac{ \tilde{p}_{st}\tilde{q}_{\mathcal I}}{\tilde{q}_{st}\tilde{q}_{st}} \right )\varepsilon, & \text{if} \quad \tilde{q}_{st}\neq \tilde{0}, \\
\dfrac{ \tilde{p}_{\mathcal I}}{\tilde{q}_{\mathcal I}}+\tilde{c}\varepsilon, & \text{if} \quad \tilde{p}_{st}=\tilde{q}_{st}=\tilde{0},~\tilde{q}_{\mathcal I}\ne\tilde{0}
\end{cases}
\end{equation*}
where $\tilde{c}\in \mathbb{Q}$ denotes an arbitrary quaternion.  If $\hat{q}\hat{p}=\hat{p}\hat{q}=\hat{1}$, then $\hat{p}$ is called invertible and the inverse of $\hat{p}$ is $\hat{p}^{-1}=\hat{q}$.

Denote the set of unit dual quaternions as $\hat{\mathbb{U}}=\{\hat{p}\in\hat{\mathbb{Q}}|~|\hat{p}|=1\}$. Let $\hat{p}\in \hat{\mathbb{U}}$, then $\hat{p}^{\ast }\hat{p}=\hat{p}\hat{p}^{\ast }=\hat{1}$. Hence, $\hat{p}$ is invertible and $\hat{p}^{-1}=\hat{p}^{\ast}$. Generally, if $\hat{p}=\tilde{p}_{st}+\tilde{p}_{\mathcal I}\varepsilon\in\hat{\mathbb{Q}}$ is appreciable, then $\hat{p}^{-1}=\tilde{p}^{-1}_{st}-\tilde{p}^{-1}_{st}\tilde{p}_{\mathcal I}\tilde{p}^{-1}_{st}\varepsilon$.

A unit dual quaternion $\hat{u}$ is called the projection of $\hat{q}\in \hat{\mathbb Q}$ onto the set $\hat{\mathbb{U}}$, if
\begin{equation*}
\hat{u}\in \underset{\hat{v}\in\hat{\mathbb{U}}}{\arg\min} ~
|\hat{v}-\hat{q}|^2.
\end{equation*}
A feasible solution was given in \cite{Cui2023}. Denote $\hat{q}=\tilde{q}_{st}+\tilde{q}_{\mathcal I}\varepsilon$. If $\tilde{q}_{st}\neq \tilde{0}$, let
\begin{equation*}
\hat{u}=\frac{\hat{q}}{\left | \hat{q}\right |}=\frac{\tilde{q}_{st}}{\left | \tilde{q}_{st}\right |}+\left ( \frac{\tilde{q}_{\mathcal I}}{\left | \tilde{q}_{st}\right |}-\frac{\tilde{q}_{st}}{\left | \tilde{q}_{st}\right |}sc\left( \frac{\tilde{q}_{st}^{\ast }}{\left | \tilde{q}_{st}\right |}\frac{\tilde{q}_{\mathcal I}}{\left | \tilde{q}_{st}\right |}\right)\right )\varepsilon.
\end{equation*}
If $\tilde{q}_{st}=\tilde{0}$ and $\tilde{q}_{\mathcal I}\neq \tilde{0}$, let
\begin{equation*}
\hat{u}=\frac{\tilde{q}_{\mathcal{I}}}{\left | \tilde{q}_{\mathcal{I}}\right |}+\tilde{c}\varepsilon \in \hat{\mathbb{U}},
\end{equation*}
where $\tilde{c}$ denotes an arbitrary quaternion that satisfies $sc(\tilde{q}_{\mathcal{I}}^{\ast }\tilde{c})=\tilde{0}$.
Then $\hat{u}$ is the projection of $\hat{q}$ onto the set $\hat{\mathbb{U}}$.

\subsection{Dual quaternion matrix}

The sets of $n\times m$-dimensional dual number matrices, dual complex matrices, quaternion matrices, and dual quaternion matrices are denoted as $\mathbb{D}^{n\times m}$,
$\mathbb{DC}^{n\times m}$, $\mathbb{Q}^{n\times m}$, and$\hat{\mathbb{Q}}^{n\times m}$, respectively.
Their zero elements are ${O}^{n\times m}$, $\hat{O}^{n\times m}$, $\tilde{\mathbf{O}}^{n\times m}$ and  $\hat{\mathbf{O}}^{n\times m}$.
Denote $\hat{I}_n$, $\tilde{\mathbf{I}}_n$ and $\hat{\mathbf{I}}_n$ as the unit element of $\mathbb{DC}^{n\times n}$, $\mathbb{Q}^{n\times n}$, and $\hat{\mathbb{Q}}^{n\times n}$, respectively.

A dual complex matrix $\hat{P}=P_{st}+P_{\mathcal I}\varepsilon \in \mathbb{DC}^{n\times m}$ has standard part $P_{st}\in \mathbb C^{n\times m}$ and dual part $P_{\mathcal I}\in \mathbb C^{n\times m}$. The transpose and conjugate of $\hat{P}=(\hat{p}_{ij})$ are $\hat{P}^{T}=(\hat{p}_{ji})$ and $\hat{P}^{\ast}=(\hat{p}_{ji}^{\ast })$, respectively. If $\hat{P}^{\ast}=\hat{P}$, then $\hat{P}$ is a dual complex Hermitian matrix. A dual complex matrix $\hat{U}\in \mathbb{DC}^{n\times n}$ is a unitary matrix, if $\hat{U}^\ast\hat{U}=\hat{U}\hat{U}^\ast=\hat{I}_n$.

The 2-norm of dual complex vector $\hat{\mathbf{\mathbf x}}=\mathbf x_{st}+\mathbf x_{\mathcal I}\varepsilon\in \mathbb{DC}^{n\times 1}$ is defined as
\begin{equation*}
\left \| \hat{\mathbf x }\right \|_{2}=\begin{cases}
\sqrt{\hat{\mathbf x }^*\hat{\mathbf x }}, & \text{if} \quad \mathbf x_{st}\neq O, \\
\|\mathbf{x}_{\mathcal{I}}\|_2\varepsilon, & \text{if} \quad \mathbf x_{st}=O.
\end{cases}
\end{equation*}

The following defines the eigenvalues and eigenvectors of dual complex matrices.
\begin{definition}
Let $\hat{P}\in \mathbb{DC}^{n\times n}$. If there exist $\hat{\lambda}\in \mathbb{DC}$ and $\hat{\mathbf x}\in \mathbb{DC}^{n\times 1}$ such that $\hat{\mathbf x}$ is appreciable and
\begin{equation*}
    \hat{P}\hat{\mathbf x}=\hat{\mathbf x}\hat{\lambda},
\end{equation*}
then $\hat{\lambda}$ and $\hat{\mathbf x}$ are eigenvalue and corresponding eigenvector of $\hat{P}$.
\end{definition}

Specifically, the dual complex Hermitian matrix with dimension $n$ has exactly $n$ dual number eigenvalues.

Suppose that $\hat{P}=P_1+P_2\varepsilon \in \mathbb{DC}^{n\times m}$,
the $F$-norm and the $F^*$-norm of $\hat{P}$ are defined by
\begin{equation*}
\| \hat{P}\|_{F}=\begin{cases}
\left \| P_1\right \|_{F}+\frac{sc(tr (P_1^{\ast }P_2))}{\left \| P_1 \right \|_{F}}\varepsilon, & \text{if} \quad P_1\neq O, \\
\left \| P_2\right \|_{F}\varepsilon, & \text{if} \quad P_1=O,
\end{cases}
\end{equation*}
where $tr(P)$ is the sum of the diagonal elements of $P$, and
\begin{equation*}
\| \hat{P} \|_{F^*}=\sqrt{\left \| P_1\right \| _{F}^2+ \left \| P_2\right \| _{F}^2\varepsilon}.
\end{equation*}

A quaternion matrix $\tilde{\mathbf Q}\in \mathbb Q^{m\times n}$ can be
expressed as $\tilde{\mathbf Q}= Q_1+ Q_2 \ii+ Q_3
\jj+ Q_4 \kk$ with $ Q_1, Q_2, Q_3, Q_4\in \mathbb{R}^{m\times n}$.
Let $P_1= Q_1+Q_2 \ii$ and $P_2= Q_3+ Q_4 \ii$, then $\tilde{\mathbf Q}$ can be rewritten as
$\tilde{\mathbf Q}= P_1+ P_2 \jj$.
The $F$-norm of a quaternion matrix $\tilde{\mathbf{Q}}=(\tilde{q}_{ij}) \in \tilde{\mathbb{Q}}^{m\times n}$ is defined as
$\| \tilde{\mathbf{Q}}\| _F=\sqrt{\sum_{ij}^{} |\tilde{q}_{ij}|^2}$.
The magnitude of a quaternion vector $\tilde{\mathbf{x}}=(\tilde{x}_{i}) \in \tilde{\mathbb{Q}}^{n\times 1}$ is defined as
$\left \| \tilde{\mathbf{x}} \right \|=\sqrt{\sum_{i=1}^{n}|\tilde{x}_{i}|^2}.$

A dual quaternion matrix $\hat{\mathbf Q}=\tilde{\mathbf Q}_{st}+\tilde{\mathbf Q}_{\mathcal I}\varepsilon \in \hat{\mathbb{Q}}^{m\times n}$ has standard part  $\tilde{\mathbf Q}_{st}\in \mathbb Q^{m\times n}$ and dual part $\tilde{\mathbf Q}_{\mathcal I}\in \mathbb Q^{m\times n}$. If $\tilde{\mathbf Q}_{st}\neq \tilde{\mathbf O}$, then $\hat{\mathbf Q}$ is called appreciable. The transpose and conjugate of $\hat{\mathbf{Q}}=(\hat{q}_{ij})$ are
$\hat{\mathbf{Q}}^{T}=(\hat{q}_{ji})$ and $\hat{\mathbf{Q}}^{\ast}=(\hat{q}_{ji}^{\ast })$, respectively.
Denote the set of unitary dual quaternion matrix with dimension $n$ as
$$
\hat{\mathbb{U}}^n_2=\{\hat{\mathbf{U}}\in \hat{\mathbb{Q}}^{n\times n}|~\hat{\mathbf{U}}^\ast\hat{\mathbf{U}}=\hat{\mathbf{U}}\hat{\mathbf{U}}^\ast=\hat{\mathbf{I}}_n\}.
$$
Denote the set of dual quaternion Hermitian matrix with dimension $n$ as
$$
\hat{\mathbb{H}}^n=\{\hat{\mathbf{Q}}\in \hat{\mathbb{Q}}^{n\times n} |~\hat{\mathbf{Q}}^{\ast}=\hat{\mathbf{Q}}\}.
$$

The $2$-norm of a dual quaternion vector  $\hat{\mathbf{x}}=\tilde{\mathbf{x}}_{st}+\tilde{\mathbf{x}}_{\mathcal{I}}\varepsilon \in \hat{\mathbb{Q}}^{n\times 1}$ is defined as
\begin{equation*}
\| \hat{\mathbf x}\|_{2}=\begin{cases}
\sqrt{\hat{\mathbf x}^*\hat{\mathbf x}}, & \text{if} \quad \tilde{\mathbf x }_{st}\neq \tilde{\mathbf O }, \\
\|\tilde{\mathbf{x}}_{\mathcal{I}}\|\varepsilon, & \text{if} \quad \tilde{\mathbf x }_{st}=\tilde{\mathbf O },
\end{cases}
\end{equation*}
and denote $\hat{\mathbb{Q}}_2^{n\times 1}=\{\hat{\mathbf{x}}\in \hat{\mathbb{Q}}^{n\times 1}|~\| \hat{\mathbf x}\|_{2}=1\}$.
The $2^{R}$-norm of the dual quaternion vector $\hat{\mathbf{x}}$ is
\begin{equation*}
\| \hat{\mathbf x}\|_{2^{R}}=\sqrt{\| \tilde{\mathbf x}_{st}  \|^{2}+\| \tilde{\mathbf x }_{\mathcal I} \|^{2}}.
\end{equation*}

The $F$-norm and $F^{R}$-norm of a dual quaternion matrix $\hat{\mathbf Q}=\tilde{\mathbf Q}_{st}+\tilde{\mathbf Q}_{\mathcal I}\varepsilon$ are defined by
\begin{equation*}
\| \hat{\mathbf Q} \|_{F}=\begin{cases}
\| \tilde{\mathbf Q}_{st} \|_{F}+\frac{sc (tr(\tilde{\mathbf Q}_{st}^{\ast }\tilde{\mathbf Q}_{\mathcal I}))}{\|  \tilde{\mathbf Q}_{st} \|_{F}}\varepsilon, & \text{if} \quad \tilde{\mathbf Q}_{st}\neq \tilde{\mathbf O }, \\
 \| \tilde{\mathbf Q}_{\mathcal I} \|_{F}\varepsilon, & \text{if} \quad \tilde{\mathbf Q}_{st}=\tilde{\mathbf O },
\end{cases}
\end{equation*}
and
\begin{equation*}
 \| \hat{\mathbf Q} \|_{F^{R}}=\sqrt{ \| \tilde{\mathbf Q}_{st} \|_{F}^{2}+\| \tilde{\mathbf Q}_{\mathcal I}\|_{F}^{2}}.
\end{equation*}
For details on these definitions, refer to \cite{Cui2023,Qi2022}.

A unit quaternion vector $\hat{\mathbf u}$ is called the projection of a dual quaternion vector $\hat{\mathbf x}$ onto the set of dual quaternion vector with unit $2$-norm, i.e., $\hat{\mathbb{Q}}_2^{n\times 1}$, if
\begin{equation*}
\hat{\mathbf{u}}\in\underset{\hat{\mathbf{v}}\in\hat{\mathbb{Q}}_2^{n\times1}}{\operatorname*{\arg\min}}~\|\hat{\mathbf{v}}-\hat{\mathbf{x}}\|_2^2.
\end{equation*}
A feasible solution was given in \cite{Cui2023}.
If $\tilde{\mathbf{x}}_{st}\neq \tilde{\mathbf{O}}^{n\times 1}$, denote $$\hat{\mathbf u}=\frac{\hat{\mathbf x}}{\left \| \hat{\mathbf x}\right \|_2}=\frac{\tilde{\mathbf x}_{st}}{\left \| \tilde{\mathbf x}_{st}\right \|}+\left ( \frac{\tilde{\mathbf x}_{\mathcal I}}{\left \| \tilde{\mathbf x}_{st}\right \|}-\frac{\tilde{\mathbf x}_{st}}{\left \| \tilde{\mathbf x}_{st}\right \|}sc\left( \frac{\tilde{\mathbf x}_{st}^{\ast }}{\left \| \tilde{\mathbf x}_{st}\right \|}\frac{\tilde{\mathbf x}_{\mathcal I}}{\left \| \tilde{\mathbf x}_{st}\right \|}\right) \right )\varepsilon.
$$ If $\tilde{\mathbf x}_{st}= \tilde{\mathbf{O}}^{n\times 1}$ and $\tilde{\mathbf x}_{\mathcal I}\neq \tilde{\mathbf{O}}^{n\times 1}$, denote
$$\hat{\mathbf u}=\frac{\tilde{\mathbf x}_{\mathcal{I}}}{\left \| \tilde{\mathbf x}_{\mathcal{I}}\right \|}+\tilde{\mathbf c}\varepsilon \in \hat{\mathbb{Q}}_2^{n\times1},$$ where $\tilde{\mathbf c}$ is arbitrary queternion vector satisfying
$sc(\tilde{\mathbf x}_{\mathcal{I}}^{\ast }\tilde{\mathbf c})=\tilde{0}$.
Then $\hat{\mathbf u}$ is the projection of $\hat{\mathbf x}$ onto the set
$\hat{\mathbb{Q}}_2^{n\times 1}$.

The following definition of eigenvalue and eigenvector of dual quaternion matrices was given in \cite{Li2023}.
\begin{definition}\label{def:righteig}
Let $\hat{\mathbf Q}\in \hat{\mathbb{Q}}^{n\times n}$.
If there exist $\hat{\lambda} \in \hat{\mathbb{Q}}$ and $\hat{\mathbf x}\in \hat{\mathbb{Q}}^{n\times 1}$ such that $\hat{\mathbf x}$ is appreciable and
\begin{equation*}
    \hat{\mathbf Q}\hat{\mathbf x}=\hat{\mathbf x}\hat{\lambda},
\end{equation*}
then $\hat{\lambda}$ is called a {\bf right eigenvalue} of $\hat{\mathbf Q}$ with $\hat{\mathbf x}$ as the corresponding {\bf right eigenvector}.

If there exist $\hat{\lambda} \in \hat{\mathbb{Q}}$ and $\hat{\mathbf x}\in \hat{\mathbb{Q}}^{n\times 1}$, where $\hat{\mathbf x}$ is appreciable, such that
\begin{equation*}
    \hat{\mathbf Q}\hat{\mathbf x}=\hat{\lambda}\hat{\mathbf x},
\end{equation*}
then we call $\hat{\lambda}$ is a {\bf left eigenvalue} of $\hat{\mathbf Q}$ with $\hat{\mathbf x}$ as an associated {\bf left eigenvector}.
\end{definition}

Let $\hat{\lambda}$ be a right eigenvalue of $\hat{\mathbf Q}$ with $\hat{\mathbf x}$ as the corresponding right eigenvector. Since a dual number is commutative with a dual quaternion vector, then if $\hat{\lambda}$ is a dual number, it is also a left eigenvalue of $\hat{\mathbf{Q}}$. In this case, we simply call $\hat{\lambda}$ an {\bf eigenvalue} of $\hat{\mathbf{Q}}$ with $\hat{\mathbf{x}}$ as an associated {\bf eigenvector}.

A dual quaternion Hermitian matrix $\hat{\mathbf{Q}}\in\hat{\mathbb{H}}^n$ has exactly $n$ eigenvalues, which are all dual numbers \cite[Theorem4.2]{Li2023}. Similar to Hermitian matrices, $\hat{\mathbf{Q}}$ also exhibits a unitary decomposition. There exist a unitary dual quaternion matrix $\hat{\mathbf{U}}\in \hat{\mathbb{U}}^{n}_2$ and a diagonal dual number matrix $\hat{\Sigma} \in \mathbb D^{n\times n}$ such that  $\hat{\mathbf{Q}}=\hat{\mathbf{U}}^*\hat{\Sigma}\hat{\mathbf{U}}$ \cite{Li2023}.

Suppose that $\{\hat{\lambda}_i=\lambda_{i,st}+\lambda_{i,\mathcal I}\}_{i=1}^{n}$ are eigenvalues of $\hat{\mathbf{Q}}\in\hat{\mathbb{H}}^n$, and $\hat{\lambda}_1$ is called the strict dominant eigenvalue \cite{Cui2023} of $\hat{\mathbf{Q}}$ with multiplicity $l$ if $$\hat{\lambda}_1=\hat{\lambda}_2=\cdots=\hat{\lambda}_l,~
|\lambda_{1,st}|>|\lambda_{l+1,st}|\ge\cdots\ge|\lambda_{n,st}|.$$

\section{Dual Complex Adjoint Matrix}\label{section3}
In this section, we introduce the dual complex adjoint matrix and establish the connection between the eigenvalue problems of dual
quaternion matrices and dual complex adjoint matrices.

We now introduce the {\it dual complex adjoint matrix} \cite{chen2024dual} of a dual quaternion matrix $\hat{\mathbf Q}=(A_1+A_2 \jj)+( A_3+ A_4 \jj)\varepsilon$, defined via the mapping $\mathcal {J}$,
\begin{equation*}\label{JJ}
\mathcal{J}(\hat{\mathbf Q})=\begin{bmatrix}
 A_1 & A_2 \\
-\overline{A_2}   & \overline{ A_1}
\end{bmatrix}
+\begin{bmatrix}
 A_3 &  A_4 \\
-\overline{ A_4}   & \overline{A_3}
\end{bmatrix} \varepsilon,
\end{equation*}
which is a bijection from the set $\hat{\mathbb{Q}}^{m\times n}$ to the set $DM(\mathbb{C}^{m\times n})$, where
\begin{equation*}
DM(\mathbb{C}^{m\times n})=
\left.\left\{\mathcal{J}(\hat{\mathbf Q})=\begin{bmatrix}
A_1 &  A_2 \\
-\overline{ A_2}   & \overline{ A_1}
\end{bmatrix}
+\begin{bmatrix}
 A_3 & A_4 \\
-\overline{A_4}   & \overline{ A_3}
\end{bmatrix} \varepsilon \right| A_1, A_2, A_3, A_4\in\mathbb{C}^{m\times n}\right\}
\end{equation*}
is the set of dual complex adjoint matrices with dimension $2m \times 2n$. Specifically, $DM(\mathbb{C}^{m \times n})$ is a subset of $\mathbb {DC}^{2m \times 2n}$.

The following lemma lists some useful properties of bijection $\mathcal{J}$, see \cite[Lemma 2]{chen2024dual}.

\begin{lemma}\label{lemma3.2}
Let $\hat{\mathbf P},\hat{\mathbf P}_1\in \hat{\mathbb{Q}}^{m\times k}$, $\hat{\mathbf Q} \in \hat{\mathbb{Q}}^{k\times n}$, $\hat{\mathbf R}\in \hat{\mathbb{Q}}^{n\times n}$, then
\begin{itemize}
\item[{\rm (i)}]  $\mathcal{J}(\hat{\mathbf O}^{m\times n})=\hat{O}^{2m\times 2n},\mathcal{J}(\hat{\mathbf I}_n)=\hat{I}_{2n}$.
\item[{\rm (ii)}]  $\mathcal{J}(\hat{\mathbf P}\hat{\mathbf Q})=\mathcal{J}(\hat{\mathbf P})\mathcal{J}(\hat{\mathbf Q})$.
\item[{\rm (iii)}]  $\mathcal{J}(\hat{\mathbf P}+\hat{\mathbf P}_1)=\mathcal{J}(\hat{\mathbf P})+\mathcal{J}(\hat{\mathbf P}_1)$.
\item[{\rm (iv)}]  $\mathcal{J}(\hat{\mathbf P}^*)=\mathcal{J}(\hat{\mathbf P})^*$.
\item[{\rm (v)}]  $\mathcal{J}(\hat{\mathbf R})$ is unitary (Hermitian) if and only if $\hat{\mathbf R}$ is unitary (Hermitian).
\item[{\rm (vi)}]  $\mathcal{J}$ is an isomorphism from ring $(\hat{\mathbb{Q}}^{n\times n},+,\cdot)$ to ring $(DM(\mathbb{C}^{n\times n}),+,\cdot)$.
\end{itemize}
\end{lemma}

For convenience, we introduce the following notation. Define the mapping $\mathcal{F}:\hat{\mathbb{Q}}^{n\times 1}\to\mathbb{DC}^{2n\times 1}$ by
\begin{equation*}
\mathcal{F}(\mathbf v_{1}+\mathbf v_{2} \jj+(\mathbf v_{3}+\mathbf v_{4} \jj))=
\begin{bmatrix}
    \mathbf v_{1}\\
    -\overline{\mathbf v_{2}}
    \end{bmatrix}+
    \begin{bmatrix}
    \mathbf v_{3}\\
    -\overline{\mathbf v_{4}}
    \end{bmatrix}\varepsilon,
\end{equation*}
which is a bijection with inverse $\mathcal{F}^{-1}:\mathbb{DC}^{2n\times 1}\to \hat{\mathbb{Q}}^{n\times 1}$ given by
\begin{equation*}
\mathcal{F}^{-1}\left ( \begin{bmatrix}
    \mathbf u_{1}\\
    \mathbf u_{2}
    \end{bmatrix}+
    \begin{bmatrix}
    \mathbf u_{3}\\
    \mathbf u_{4}
    \end{bmatrix}\varepsilon \right )=
    \mathbf u_{1}-\overline{\mathbf u_{2}} \jj+
(\mathbf u_{3}-\overline{\mathbf u_{4}} \jj  ) \varepsilon,
\end{equation*}
where $\mathbf u_{1},\mathbf u_{2},\mathbf u_{3},\mathbf u_{4}
\in \mathbb C^{n\times 1}$.

As shown in \cite[Lemma 3]{chen2024dual}, the study of right eigenvalues for dual quaternion matrices can be simplified by focusing on their representative right eigenvalues, represented as dual complex numbers.

Building on this foundation, we now explore the relationship between dual complex adjoint matrices and dual quaternion matrices in the context of the eigenvalue problem.

\begin{theorem}\label{theorem3.1}
Let $\hat{\mathbf Q}\in\hat{\mathbb{Q}}^{n\times n}$, $\hat{\mathbf v}\in\hat{ \mathbb{Q}}^{n\times 1}$ and $\hat{\lambda}=\lambda_{1}+\lambda_{2}\varepsilon\in \mathbb {DC}$. Denote
$\hat{P}=\mathcal{J}(\hat{\mathbf Q})$,
$\hat{\mathbf u}_1=\mathcal{F}(\hat{\mathbf v})$, and
$\hat{\mathbf u}_2=\mathcal{F}(\hat{\mathbf v} \jj)$,
then
\begin{equation*}
\hat{\mathbf Q}\hat{\mathbf v}=\hat{\mathbf v}\hat{\lambda}
\end{equation*}
is equivalent to
\begin{equation*}
\hat{P}\hat{\mathbf u}_1=\hat{\lambda}\hat{\mathbf u}_1~\text{or}~\hat{P}\hat{\mathbf u}_2=\overline{\hat{\lambda}}\hat{\mathbf u}_2.
\end{equation*}
In addition, $\hat{\mathbf u}_1$ and $\hat{\mathbf u}_2$ are orthogonal, i.e., $\hat{\mathbf u}^*_1\hat{\mathbf u}_2=\hat{\mathbf u}^*_2\hat{\mathbf u}_1=\hat{0}$.
\end{theorem}
\begin{proof}
First, note that $$\mathcal{J}(\hat{\mathbf v})=[\mathcal{F}(\hat{\mathbf v})~-\mathcal{F}(\hat{\mathbf v}\jj)]=[\hat{\mathbf u}_1~-\hat{\mathbf u}_2].$$
Since the mapping $\mathcal{J}$ is a bijection, then it follows from (i) and (ii) in Lemma \ref{lemma3.2} that
\begin{align*}
\hat{\mathbf Q}\hat{\mathbf v}=\hat{\mathbf v}\hat{\lambda}&\Longleftrightarrow \mathcal{J}(\hat{\mathbf Q})\mathcal{J}(\hat{\mathbf v})=
\mathcal{J}(\hat{\mathbf v})\mathcal{J}(\hat{\lambda}\hat{\mathbf I}_n)
\\&\Longleftrightarrow \mathcal{J}(\hat{\mathbf Q})
[\mathcal{F}(\hat{\mathbf v})~-\mathcal{F}(\hat{\mathbf v}\jj)]=
[\mathcal{F}(\hat{\mathbf v})~-\mathcal{F}(\hat{\mathbf v}\jj)]\operatorname{diag}
(\hat{\lambda}\hat{I}_n~\overline{\hat{\lambda}}\hat{I}_n)
\\&\Longleftrightarrow\hat{P}[\hat{\mathbf u}_1~-\hat{\mathbf u}_2]=[\hat{\lambda}\hat{\mathbf u}_1~-\overline{\hat{\lambda}}\hat{\mathbf u}_2].
\end{align*}
Therefore, if
$\hat{\mathbf{Q}}\hat{\mathbf v}=\hat{\mathbf v}\hat{\lambda}$, then it follows that $\hat{P}\hat{\mathbf u}_1=\hat{\lambda}\hat{\mathbf u}_1$ and $\hat{P}\hat{\mathbf u}_2=\overline{\hat{\lambda}}\hat{\mathbf u}_2$.

Conversely, if $\hat{P}\hat{\mathbf u}_1=\hat{\lambda}\hat{\mathbf u}_1$, suppose that
$$\hat{P}=\mathcal{J}(\hat{\mathbf{Q}})=\begin{bmatrix}
\hat{P}_1  & \hat{P}_2 \\
-\overline{\hat{P}_2} & \overline{\hat{P}_1}
\end{bmatrix}, ~\hat{\mathbf u}_1=\mathcal{F}(\hat{\mathbf v})=\begin{bmatrix}
\hat{\mathbf v}_1 \\
-\overline{\hat{\mathbf v}_2}
\end{bmatrix}.$$ Then the eigenvalue equation implies
$$\hat{P}_1\hat{\mathbf v}_1-\hat{P}_2\overline{\hat{\mathbf v}_2}=\hat{\lambda}\hat{\mathbf v}_1, \quad
-\overline{\hat{P}_2}\hat{\mathbf v}_1 - \overline{\hat{P}_1}\overline{\hat{\mathbf v}_2}=-\hat{\lambda}\overline{\hat{\mathbf v}_2}.$$
Taking conjugates and rearranging terms, we obtain
$$\hat{P}_1\hat{\mathbf v}_2+\hat{P}_2\overline{\hat{\mathbf v}_1}=\overline{\hat{\lambda}}\hat{\mathbf v}_2, \quad
-\overline{\hat{P}_2}\hat{\mathbf v}_2 +\overline{\hat{P}_1}\overline{\hat{\mathbf v}_1}=\overline{\hat{\lambda}}\overline{\hat{\mathbf v}_1}.$$
Therefore, $\mathcal{J}(\hat{\mathbf{Q}})\mathcal{F}(\hat{\mathbf v}\jj)=\overline{\hat{\lambda}}\mathcal{F}(\hat{\mathbf v}\jj)$, i.e., $\hat{P}\hat{\mathbf u}_2=\overline{\hat{\lambda}}\hat{\mathbf u}_2$.
Consequently, we have $\hat{P}[\hat{\mathbf u}_1~-\hat{\mathbf u}_2]=[\hat{\lambda}\hat{\mathbf u}_1~-\overline{\hat{\lambda}}\hat{\mathbf u}_2]$, from which it follows that $\hat{\mathbf{Q}}\hat{\mathbf v}=\hat{\mathbf v}\hat{\lambda}$.
A similar argument shows that if  $\hat{P}\hat{\mathbf u}_2=\hat{\lambda}\hat{\mathbf u}_2$, then
$\hat{\mathbf{Q}}\hat{\mathbf v}=\hat{\mathbf v}\hat{\lambda}$ also holds.

We now show that $\hat{\mathbf u}_1$ and $\hat{\mathbf u}_2$ are orthogonal.
Suppose that $\hat{\mathbf v}=\mathbf v_{st,1}+\mathbf v_{st,2}\jj+(\mathbf v_{\mathcal I,1}+\mathbf v_{\mathcal I,2}\jj)\varepsilon$. Since
\begin{align*}
(\hat{\mathbf u}^*_1\hat{\mathbf u}_2)_{st}
&=\begin{bmatrix}
\mathbf v_{st,1}\\
-\overline{\mathbf v_{st,2}}
\end{bmatrix}^*
\begin{bmatrix}
-\mathbf v_{st,2}\\
-\overline{\mathbf v_{st,1}}
\end{bmatrix}
=-\overline{\mathbf v_{st,1}}^T\mathbf v_{st,2}+\mathbf v_{st,2}^T \overline{\mathbf v_{st,1}}=\tilde{0},
\end{align*}
\begin{align*}
(\hat{\mathbf u}^*_1\hat{\mathbf u}_2)_{\mathcal I}
&=\begin{bmatrix}
\mathbf v_{\mathcal I,1}\\
-\overline{\mathbf v_{\mathcal I,2}}
\end{bmatrix}^*
\begin{bmatrix}
-\mathbf v_{st,2}\\
-\overline{\mathbf v_{st,1}}
\end{bmatrix}
+\begin{bmatrix}
\mathbf v_{st,1}\\
-\overline{\mathbf v_{st,2}}
\end{bmatrix}^*
\begin{bmatrix}
-\mathbf v_{\mathcal I,2}\\
-\overline{\mathbf v_{\mathcal I,1}}
\end{bmatrix}
\\&=-\overline{\mathbf v_{\mathcal I,1}}^T\mathbf v_{st,2}+\mathbf v_{\mathcal I,2}^T \overline{\mathbf v_{st,1}}
-\overline{\mathbf v_{st,1}}^T\mathbf v_{\mathcal I,2}+\mathbf v_{st,2}^T \overline{\mathbf v_{\mathcal I,1}}=\tilde{0},
\end{align*}
it follows that $\hat{\mathbf u}^*_1\hat{\mathbf u}_2=\hat{0}$. Similarly, we have
$\hat{\mathbf u}^*_2\hat{\mathbf u}_1=\hat{0}$. Therefore, $\hat{\mathbf u}_1$ and $\hat{\mathbf u}_2$ are orthogonal.
\qed\end{proof}

Since $\hat{\mathbf Q}\in\hat{\mathbb{H}}^n$ has exactly $n$ dual number eigenvalues, we can derive the following corollary from Theorem\ref{theorem3.1}.
\begin{corollary}\label{corollary3.1}
Let $\hat{\mathbf Q}\in\hat{\mathbb{H}}^n$, and let $\{ \hat{\lambda}_{i}\}^{n}_{i=1}$ and $\{\hat{\mathbf v}^{i}\}^{n}_{i=1}$ be eigenvalues and corresponding eigenvectors of $\hat{\mathbf Q}$, respectively. Then the set
$$\{\hat{\lambda}_1,\hat{\lambda}_1,\hat{\lambda}_2,\hat{\lambda}_2,\ldots,\hat{\lambda}_n,\hat{\lambda}_n\}$$ consists of all the eigenvalues of $\hat{P}=\mathcal{J}(\hat{\mathbf Q})$. Moreover, $\mathcal{F}(\mathbf v^{i})$ and $\mathcal{F}(\mathbf v^{i}\jj)$ are two orthogonal eigenvectors of $\hat{P}$ with respect to the eigenvalue $\hat{\lambda}_{i}$.

Conversely, if $\hat{\mathbf u}^i\in \mathbb{DC}^{2n\times 1}$ is an eigenvector of $\hat{P}$ corresponding to the eigenvalue $\hat{\lambda}_{i}$, then $\mathcal{F}^{-1}(\hat{\mathbf u}^i)$ is an eigenvector of $\hat{\mathbf Q}$ associated with the same eigenvalue $\hat{\lambda}_{i}$.
\end{corollary}

It follows from Corollary \ref{corollary3.1} that an eigenvector $\hat{\mathbf v}$ of $\hat{\mathbf Q}\in \hat{\mathbb{H}}^n$ corresponds to two orthogonal eigenvectors $\mathcal{F}(\hat{\mathbf v})$ and $\mathcal{F}(\hat{\mathbf v} \jj)$ of $\mathcal{J}(\hat{\mathbf Q})$. We define the transformation between these two eigenvectors as follows:
\begin{equation*}
\mathcal{H}\left(\begin{bmatrix}
    \mathbf u_1\\
    \mathbf u_2
    \end{bmatrix}+
    \begin{bmatrix}
    \mathbf u_3\\
    \mathbf u_4
    \end{bmatrix}\varepsilon\right)
=\begin{bmatrix}
    \overline{\mathbf u_2}\\
    -\overline{\mathbf u_1}
    \end{bmatrix}+
    \begin{bmatrix}
    \overline{\mathbf u_4}\\
    -\overline{\mathbf u_3}
    \end{bmatrix}\varepsilon,
\end{equation*}
where $\mathbf u_{1},\mathbf u_{2},\mathbf u_{3},\mathbf u_{4}
\in \mathbb C^{n\times 1}$. Then we have $\mathcal{H}(\mathcal{F}(\hat{\mathbf v}))=\mathcal{F}(\hat{\mathbf v} \jj)$ and $\mathcal{H}(\mathcal{F}(\hat{\mathbf v} \jj))=-\mathcal{F}(\hat{\mathbf v})$.

\section{Dual Complex Adjoint Matrix based New Power Method}\label{section5}
In this section, based on the relationship between dual complex adjoint matrices and dual quaternion matrices in the context of the
eigenvalue problem established in the previous section, we propose an improved power method for computing eigenvalues of dual quaternion Hermitian matrices, referred to as DCAM-PM.
Furthermore, by incorporating Aitken extrapolation into DCAM-PM, we develop an enhanced version, referred to as ADCAM-PM, and demonstrate that it achieves a faster convergence rate than the standard power method.

By Corollary \ref{corollary3.1}, for any $\hat{\mathbf Q}\in \hat{\mathbb H}^n$, computing its strict dominant eigenpair can be reduced to find the strict dominant eigenpair of its dual complex adjoint matrix $\mathcal{J}(\hat{\mathbf Q})$. Specifically, by obtaining the dominant eigenpair $\hat{\lambda}$ and $\hat{\mathbf u}$ of $\mathcal{J}(\hat{\mathbf Q})$, one can deduce that $\hat{\lambda}$ and $\mathcal{F}^{-1}(\hat{\mathbf u})$ form the strict dominant eigenpair of $\hat{\mathbf Q}$. Based on this result, we propose an improved power method (Algorithm \ref{Power_method_algorithm}).

\begin{algorithm}[h]
    \caption{DCAM-PM} \label{Power_method_algorithm}
    \begin{algorithmic}
    \REQUIRE $\hat{\mathbf Q}\in \hat{\mathbb H}^n$, initial iteration vector $\hat{\mathbf v}^{(0)}\in \hat{\mathbb Q}^{n\times 1}_2$, the maximal iteration number $k_{max}$ and the tolerance $\delta$.
    \ENSURE $\hat{\mathbf v}^*$ and $\hat{\lambda}^{( k ) }$.
    \STATE Compute $\hat{P}=\mathcal{J}(\hat{\mathbf Q})$ and $\hat{\mathbf u}^{(0)}=\mathcal{F}(\hat{\mathbf v}^{(0)})$.
    \FOR {$k = 1$ to $k_{max}$}
    \STATE Update $\hat{\mathbf y}^{(k)}=\hat{P}\hat{\mathbf u}^{(k-1)}$, $\hat{\lambda}^{(k)}=(\hat{\mathbf u}^{(k-1)})^*\hat{\mathbf y}^{(k)}$, $\hat{\mathbf u}^{(k)}=\frac{\hat{\mathbf y}^{( k)}}{\| \hat{\mathbf y}^{(k) } \|_2 }$.
    \STATE If $\| \hat{\mathbf y}^{(k) } -\hat{\mathbf u}^{(k-1) }\hat{\lambda}^{( k ) } \| _{2^R}\le \delta$, then Stop.
    \ENDFOR
    \STATE Compute $\hat{\mathbf v}^*=\mathcal{F}^{-1}(\hat{\mathbf u}^{(k)})$.
    \end{algorithmic}
\end{algorithm}

Since DCAM-PM essentially applies the power method to compute the dominant eigenpair of a dual complex adjoint matrix, we conduct a convergence analysis tailored to dual complex Hermitian matrices. Inspired by \cite[Theorem 4.1]{Cui2023}, we state the corresponding convergence theorem concisely here, omitting the detailed proof for brevity and to avoid redundancy.

\begin{theorem}\label{Theorem2}
Let $\hat{P}=P_1+P_2\varepsilon \in \mathbb{DC}^{n\times n}$ be a dual complex Hermitian matrix, suppose that $P_1\ne O$, $\{\hat{\lambda}_i=\lambda_{i,st}+\lambda_{i,\mathcal I}\}_{i=1}^{n}$ are eigenvalues of $\hat{P}$, and $\hat{\lambda}_1$ is the strict dominant eigenvalue of $P_1$  with multiplicity $l$, i.e., $$\hat{\lambda}_1=\hat{\lambda}_2=\cdots=\hat{\lambda}_l,~
|\lambda_{1,st}|>|\lambda_{l+1,st}|\ge\cdots\ge|\lambda_{n,st}|.$$
Let $\hat{\mathbf u}_1,\hat{\mathbf u}_2,\cdots,\hat{\mathbf u}_n$ be corresponding orthogonal eigenvalues. Suppose that the initial iteration vector satisfies $\hat{\mathbf u}^{(0)}=\sum_{i=1}^{n}\hat{\mathbf u}_i\hat{\alpha}_i\in\hat{\mathbb Q}^{n\times 1}_2$ and $\sum_{j=1}^l|\tilde{\alpha}_{j,st}|\neq 0$, where $\tilde{\alpha}_{j,st}$ is the standard part of $\hat{\alpha}_{j}$, the sequence obtained by Algorithm \ref{Power_method_algorithm} satisfies
\begin{equation*}
\hat{\mathbf{u}}^{(k)}=s^k\sum_{j=1}^l\hat{\mathbf{u}}_j\hat{\gamma}_j(\hat{1}+\tilde{O}_D(|\lambda_{l+1,st}/ \lambda_{1,st}|^k)),
\end{equation*}
where $s=\operatorname{sgn}(\lambda_{1,st})$,
$\hat{\gamma}_j=\hat{\alpha}_j/\sqrt{\sum_{i=1}^l|\hat{\alpha}_i|^2}$, and
\begin{equation*}
\hat{\lambda}^{(k)}=\hat{\lambda}_1(\hat{1}+\tilde{O}_D(|\lambda_{l+1,st}/ \lambda_{1,st}|^{2k})).
\end{equation*}
\end{theorem}

To extend DCAM-PM for computing all eigenpairs of dual quaternion Hermitian matrices, we first introduce the following lemma.

\begin{lemma}\label{lemma3.7}
Let $\hat{\mathbf Q}\in \hat{\mathbb{H}}^n$, $\hat{\lambda}\in \mathbb D$ and $\hat{\mathbf v}\in \hat{\mathbb Q}^{n\times 1}_2$ is an eigenpair of $\hat{\mathbf Q}$, then
\begin{equation*}
\mathcal{J}(\hat{\mathbf Q}-\hat{\lambda}\hat{\mathbf v}\hat{\mathbf v}^*)=\mathcal{J}(\hat{\mathbf Q})-\hat{\lambda}\mathcal{F}(\hat{\mathbf v})\mathcal{F}(\hat{\mathbf v})^*-\hat{\lambda}\mathcal{H}(\mathcal{F}(\hat{\mathbf v}))\mathcal{H}(\mathcal{F}(\hat{\mathbf v}))^*.
\end{equation*}
\end{lemma}
\begin{proof}
From the identity $\mathcal{J}(\hat{\mathbf v})=[\mathcal{F}(\hat{\mathbf v})~-\mathcal{H}(\mathcal{F}(\hat{\mathbf v}))]$
and properties (ii), (iii), and (iv) in Lemma \ref{lemma3.2}, it follows that
\begin{align*}
\mathcal{J}(\hat{\mathbf Q}-\hat{\lambda}\hat{\mathbf v}\hat{\mathbf v}^*)&=\mathcal{J}(\hat{\mathbf Q})-\mathcal{J}(\hat{\lambda}\hat{\mathbf I})\mathcal{J}(\hat{\mathbf v})\mathcal{J}(\hat{\mathbf v})^*.
\\&=\mathcal{J}(\hat{\mathbf Q})-\hat{\lambda}[\mathcal{F}(\hat{\mathbf v})~-\mathcal{H}(\mathcal{F}(\hat{\mathbf v}))][\mathcal{F}(\hat{\mathbf v})~-\mathcal{H}(\mathcal{F}(\hat{\mathbf v}))]^*
\\&=\mathcal{J}(\hat{\mathbf Q})-\hat{\lambda}\mathcal{F}(\hat{\mathbf v})\mathcal{F}(\hat{\mathbf v})^*-\hat{\lambda}\mathcal{H}(\mathcal{F}(\hat{\mathbf v}))\mathcal{H}(\mathcal{F}(\hat{\mathbf v}))^*.
\end{align*}
The final equation is based on the orthogonality of $\mathcal{F}(\hat{\mathbf v})$ and $\mathcal{H}(\mathcal{F}(\hat{\mathbf v}))$.
\qed\end{proof}

To compute all eigenvalues of the dual quaternion Hermitian matrix $\hat{\mathbf Q}$ by DCAM-PM, once the strictly dominant eigenvalue $\hat{\lambda}^k$ and its corresponding eigenvector $\hat{\mathbf u}^k\in \hat{\mathbb Q}^{n\times 1}_2$ for the current dual complex adjoint matrix $\hat{P}_{k}$ are obtained, Lemma \ref{lemma3.7} guarantees that the next eigenpair of $\hat{\mathbf Q}$ can be found by simply updating $\hat{P}_{k}$ to $\hat{P}_{k+1}$ via $\hat{P}_{k+1}=\hat{P}_k-\hat{\lambda}^k\hat{\mathbf u}^k(\hat{\mathbf u}^{k})^*-\hat{\lambda}^k\mathcal{H}(\hat{\mathbf u}^k)\mathcal{H}(\hat{\mathbf u}^{k})^*$, and then applying DCAM-PM to compute the strict dominant eigenpair of $\hat{P}_{k+1}$. The complete procedure is detailed in Algorithm \ref{algorithm3.2}.

\begin{algorithm}[!h]
    \caption{DCAMA-PM}\label{algorithm3.2}
    \begin{algorithmic}
    \REQUIRE
    $\hat{\mathbf Q}\in \hat{\mathbb H}^n$, dimension $n$ and tolerance $\gamma$.
    \ENSURE $\{ \hat{\lambda}^i\}^{n}_{i=1}$ and $\{ \hat{\mathbf w}^i\}^{n}_{i=1}$.
    \STATE Compute $\hat{P}=\mathcal{J}(\hat{\mathbf Q})$. Let $\hat{P}_1=\hat{P}$.
    \FOR {$k = 1$ to $n$}
    \STATE If $\| \hat{P}_{k}\|_{F^R}\le \gamma$, then Stop.
    \STATE Compute the strict dominant eigenpair $\hat{\lambda}^k$ and $\hat{\mathbf u}^k$ of $\hat{P}_k$ by power method.
    \STATE Compute $\hat{\mathbf v}^k=\mathcal{H}(\hat{\mathbf u}^k)$ and $\hat{\mathbf w}^k=\mathcal{F}^{-1}(\hat{\mathbf u}^k)$.
    \STATE Update $\hat{P}_{k+1}=\hat{P}_k-\hat{\lambda}^k\hat{\mathbf u}^k(\hat{\mathbf u}^{k})^*-\hat{\lambda}^k\hat{\mathbf v}^k(\hat{\mathbf v}^{k})^*$.
    \ENDFOR
    \end{algorithmic}
\end{algorithm}

Since the convergence rate of the power method depends on the distribution of eigenvalues (see Theorem \ref{Theorem2}), it can sometimes converge slowly. To address this, we accelerate convergence by incorporating Aitken extrapolation into DCAM-PM. Aitken extrapolation is an effective technique for enhancing the convergence of iterative sequences. Given a sequence $\{x_k\}_{k=1}^{\infty}$, the Aitken iterative sequence $\{y_k\}_{k=1}^{\infty}$ is defined by
$$y_k=x_k-\frac{(x_{k+1}-x_k)^2}{x_{k+2}+x_k-2x_{k+1}}.$$

The following lemma demonstrates that if the original sequence converges linearly, the Aitken iterative sequence achieves a faster convergence rate.

\begin{lemma}\label{lem:Aitken}
Let $x_k=A+Bq^k+o(q^k),~k=1,2,\ldots$, where $0<q<1$ and $A,B\in \mathbb{R}$ are constants. Denote $$y_k=x_k-\frac{(x_{k+1}-x_k)^2}{x_{k+2}+x_k-2x_{k+1}}.$$ Then we have
$$\lim_{k \to \infty} \frac{|y_k-A|}{|x_k-A|}=0.$$
\end{lemma}
\begin{proof}
Since
\begin{align*}
y_k-A&=x_k-\frac{(x_{k+1}-x_n)^2}{x_{k+2}+x_k-2x_{k+1}}-A
\\&=Bq^k-\frac{(Bq^k(q-1)+o(q^k))^2}
{B(q^k-2q^{k+1}+q^{k+2})+o(q^k)}=o(q^k),
\end{align*}
then we have $$\lim_{k \to \infty} \frac{|y_k-A|}{|x_k-A|}=\lim_{k \to \infty}\frac{o(q^k)}{q^k}=0.$$
\end{proof}

We now apply the Aitken extrapolation to the power method. Given a dual quaternion Hermitian matrix $\hat{\mathbf Q}\in \hat{\mathbb{H}}^n$. Suppose that $\{\lambda^{(k)}\}_{k=1}^{\infty}$ and
$\{\mathbf{\hat{v}}^{(k)}\}_{k=1}^{\infty}$ are the iterative sequences for the strict dominant eigenvalue and corresponding eigenvector of $\hat{\mathbf Q}$ generated by the power method, respectively. We apply Aitken extrapolation to the standard part and dual part of the sequence, respectively, i.e., let
\begin{subequations}\label{equ1}
\begin{align}
\mathbf{\tilde{w}}^{(k)}_{st,i}&=\mathbf{\tilde{v}}^{(k)}_{st,i}-\frac{(\mathbf{\tilde{v}}^{(k+1)}_{st,i}-\mathbf{\tilde{v}}^{(k)}_{st,i})^2}{\mathbf{\tilde{v}}^{(k+2)}_{st,i}+\mathbf{\tilde{v}}^{(k)}_{st,i}-2\mathbf{\tilde{v}}^{(k+1)}_{st,i}}, \label{equ1_1}
\\\mathbf{\tilde{w}}^{(k)}_{\mathcal I,i}&=\mathbf{\tilde{v}}^{(k)}_{\mathcal I,i}-\frac{(\mathbf{\tilde{v}}^{(k+1)}_{\mathcal I,i}-\mathbf{\tilde{v}}^{(k)}_{\mathcal I,i})^2}{\mathbf{\tilde{v}}^{(k+2)}_{\mathcal I,i}+\mathbf{\tilde{v}}^{(k)}_{\mathcal I,i}-2\mathbf{\tilde{v}}^{(k+1)}_{\mathcal I,i}}, \label{equ1_2}
\end{align}
\end{subequations}
where $\mathbf{\tilde{w}}^{(k)}_{st,i}$ and $\mathbf{\tilde{w}}^{(k)}_{\mathcal I,i}$ denote the i-th components of $\mathbf{\tilde{w}}^{(k)}_{st}$ and $\mathbf{\tilde{w}}^{(k)}_{\mathcal I}$, respectively, and let
\begin{subequations}\label{equ2}
\begin{align}
\kappa^{(k)}_{st}&=\lambda^{(k)}_{st}-\frac{(\lambda^{(k+1)}_{st}-\lambda^{(k)}_{st})^2}{\lambda^{(k+2)}_{st}+\lambda^{(k)}_{st}-2\lambda^{(k+1)}_{st}}, \label{equ2_1}
\\\kappa^{(k)}_{\mathcal I}&=\lambda^{(k)}_{\mathcal I}-\frac{(\lambda^{(k+1)}_{\mathcal I}-\lambda^{(k)}_{\mathcal I})^2}{\lambda^{(k+2)}_{\mathcal I}+\lambda^{(k)}_{\mathcal I}-2\lambda^{(k+1)}_{\mathcal I}}. \label{equ2_2}
\end{align}
\end{subequations}
We refer $\{\kappa^{(k)}\}_{k=1}^{\infty}$ and
$\{\mathbf{\hat{w}}^{(k)}\}_{k=1}^{\infty}$ to the Aitken iterative sequences for $\{\lambda^{(k)}\}_{k=1}^{\infty}$ and
$\{\mathbf{\hat{v}}^{(k)}\}_{k=1}^{\infty}$.

The following theorem shows that the sequences $\{\mathbf{\hat{w}}^{(k)}\}_{k=1}^{\infty}$ and $\{\kappa^{(k)}\}_{k=1}^{\infty}$ converge to the strict dominant eigenvalue and corresponding eigenvector of $\hat{\mathbf Q}$, respectively, at a rate faster than that of the power method.

\begin{theorem}
Given $\hat{\mathbf Q}\in \hat{\mathbb{H}}^n$ and $\mathbf{\hat{v}}^{(0)}\in\hat{\mathbb{Q}}^{n\times 1}_2$, where $\tilde{\mathbf{Q}}_{\text{st}}\neq\tilde{\mathbf{O}}_{n\times n}$. Suppose that $\{\lambda_j\}_{j=1}^{n}$ are eigenvalues of $\hat{\mathbf Q}$ and $\{\mathbf{\hat{u}}_j\}_{j=1}^{n}$ are corresponding orthonormal eigenvectors. Suppose that $\mathbf{\hat{v}}^{(0)}=\sum_{j=1}^n\mathbf{\hat{u}}_j\mathbf{\hat{\alpha}}_j$, where $\mathbf{\hat{\alpha}}_j\in\hat{\mathbb{Q}}$ for $j=1,2,\ldots,n$. Let
$\{\lambda^{(k)}\}_{k=1}^{\infty}$ and
$\{\mathbf{\hat{v}}^{(k)}\}_{k=1}^{\infty}$ are iterative sequences for strict dominant eigenvalue and eigenvector of $\hat{\mathbf Q}$ generated by power method, respectively. Suppose that the eigenvalues of $\hat{\mathbf Q}$ satisfy $\lambda_{l_j+1}=\lambda_{l_j+2}=\cdots=\lambda_{l_{j+1}}$ for $j=0,1,\cdots,m$, where $0=l_0<l_1<\cdots<l_{m+1}=n$, and $|\lambda_{l_1,st}|>|\lambda_{l_2,st}|\ge\ldots\ge|\lambda_{l_{m+1},st}|$. Suppose that $\sum_{j=1}^{l_1}|\tilde{\alpha}_{j,\text{st}}|\neq0$. Let $\{\kappa^{(k)}\}_{k=1}^{\infty}$ and
$\{\mathbf{\hat{w}}^{(k)}\}_{k=1}^{\infty}$ be the Aitken iterative sequences generated from $\{\lambda^{(k)}\}_{k=1}^{\infty}$ and
$\{\mathbf{\hat{v}}^{(k)}\}_{k=1}^{\infty}$ by \eqref{equ1} and \eqref{equ2}.
Then $\{\kappa^{(k)}\}_{k=1}^{\infty}$ converges to
$\lambda_1$ and $\{s^j\hat{\mathbf{w}}^{(k)}\}_{k=1}^{\infty}$ converges to
$\sum_{j=1}^{l_1}\mathbf{\hat{u}}_j\hat{\gamma}_j$ with a faster rate than power method, where $s=sgn(\lambda_{1,\mathrm{st}})$, $\hat{\gamma}_j=\frac{\hat{\alpha}_j}{\sqrt{\sum_{i=1}^{l_1}|\hat{\alpha}_i|^2}}$.
\end{theorem}
\begin{proof}
By the process of power method, we have $$\mathbf{\hat{v}}^{(k)}=\frac{\sum_{j=1}^n\hat{\mathbf{u}}_j\lambda_j^k\hat{\alpha}_j}{\sqrt{\sum_{j=1}^n|\lambda_j^k\hat{\alpha}_j|^2}}\quad {\rm and}\quad \lambda^{(k)}=\frac{\sum_{j=1}^n\lambda_j^{2k+1}|\mathbf{\hat{\alpha}}_j|^2}{\sum_{j=1}^n|\lambda_j^k\mathbf{\hat{\alpha}}_j|^2}.$$

Suppose that $\lambda_{1+l_{j-1}}=\mu_j+\eta_j\varepsilon$ for $j=1,\ldots, m+1$. By direct derivation, we obtain $|\lambda_{j}^k\hat{\alpha}_{j}|=|\lambda_{j}^k||\hat{\alpha}_{j}|$. Furthermore, we have $\lambda_{1+l_{j-1}}^k=\mu_{j}^k+k\mu_{j}^{k-1}\eta_{j}\epsilon=\mu_j^k(1+k\beta_j\epsilon)$ and $|\lambda_{1+l_{j-1}}^k|=|\mu_{j}|^k(1+k\beta_j\epsilon)$, where $\beta_j=\mu_{j}^{-1}\eta_j$. 

Suppose that $|\mu_1|>|\mu_2|=\ldots=|\mu_s|>|\mu_{s+1}|\ge\ldots\ge|\mu_{m+1}|$. Denote $\sum_{k=1+l_{j-1}}^{l_j}|\hat{\alpha}_k|^2=x_j+y_j\varepsilon$, for $j=1,2,\ldots,m+1$. Denote $\frac{\lambda_{1+l_{j-1}}}{\lambda_1}=u_j+v_j\varepsilon$, for $j=2,3,\ldots,s$. Denote $q=\frac{|\mu_2|}{|\mu_1|}$.
Then we have
\begin{align*}
\sum_{j=1}^n\lambda_j^{2k+1}|\hat{\alpha}_j|^2&=
\sum_{j=1}^{m+1}\lambda_j|\mu_j|^{2k}(1+2k\beta_j\varepsilon)(x_j+y_j\varepsilon)
\\&=\lambda_1|\mu_1|^{2k}(x_1+q^{2k}\sum_{t=2}^{s}u_tx_t+o(q^{2k}))
+\lambda_1|\mu_1|^{2k}(2k\beta_1x_1+y_1\\&+q^{2k}(\sum_{t=2}^{s}x_tv_t+2k\sum_{t=2}^{s}\beta_tx_tu_t
+\sum_{t=2}^{s}y_tu_t)+o(q^{2k}))\varepsilon,
\end{align*}
and
\begin{align*}
\sum_{j=1}^n|\lambda_j^k\hat{\alpha}_j|^2&=
\sum_{j=1}^{m+1}|\mu_i|^{2k}(1+2k\beta_j\varepsilon)(x_j+y_j\varepsilon)
\\&=|\mu_1|^{2k}(x_1+q^{2k}\sum_{t=2}^{s}x_t+o(q^{2k}))
\\&+|\mu_1|^{2k}(2k\beta_1x_1+y_1+q^{2k}(2k\sum_{t=2}^{s}\beta_tx_t+\sum_{t=2}^{s}y_t)+o(q^{2k}))\varepsilon.
\end{align*}
Then
\begin{align*}
\frac{\sum_{j=1}^n\lambda_j^{2k+1}|\mathbf{\hat{\alpha}}_j|^2}{\sum_{j=1}^n|\lambda_j^k\mathbf{\hat{\alpha}}_j|^2}
&=\lambda_1(1+Aq^{2k}+o(q^{2k})+(Bq^{2k}+o(q^{2k}))\varepsilon)
\\&=\mu_1+\mu_1Aq^{2k}+o(q^{2k})+(\eta_1+(\eta_1A+\mu_1B)q^{2k}+o(q^{2k}))\varepsilon,
\end{align*}
where $$A=\sum_{t=2}^{s}\frac{x_t}{x_1}(u_t-1),B=\sum_{t=2}^{s} \frac{x_tv_t}{x_1}+\sum_{t=2}^{s}(u_t-1)(\frac{y_tx_1-x_ty_1}{x_1^2}+\frac{2kx_t(\beta_t-\beta_1)}{x_1}),$$ which are constants related only to $\{\lambda_j\}$ and $\{\hat{\alpha}_j\}$.
Then, by Lemma \ref{lem:Aitken}, the sequence $\{\kappa^{(k)}\}_{k=1}^{\infty}$ converges to $\lambda_1$ at a faster rate  than that of the power method.

Similarly, it can be shown that the sequence $\{s^j\hat{\mathbf{w}}^{(k)}\}_{k=1}^{\infty}$ converges to $\sum_{j=1}^{l_1}\mathbf{\hat{u}}_j\hat{\gamma}_j$ at a faster rate than that of the power method.
\end{proof}

Similarly, Aitken extrapolation can be applied to the DCAM-PM algorithm, and the theoretical results remain consistent. To reduce computational cost, the Aitken extrapolation can be deferred until the iteration sequence reaches a specified accuracy. The algorithm is structured as follows.

\begin{algorithm}
    \caption{ADCAM-PM} \label{APower_method_algorithm}
    \begin{algorithmic}
    \REQUIRE $\hat{\mathbf Q}\in \hat{\mathbb H}^n$, initial iteration vector $\hat{\mathbf u}^{(0)}\in \hat{\mathbb Q}^{n\times 1}_2$, the maximal iteration number $k_{max}$, parameter $\gamma$ and the tolerance $\delta$.
    \ENSURE $\hat{\mathbf u}^*$ and $\hat{\kappa}^{(k-2)}$.
    \STATE Compute $\hat{P}=\mathcal{J}(\hat{\mathbf Q})$ and $\hat{\mathbf v}^{(0)}=\mathcal{F}(\hat{\mathbf u}^{(0)})$.
    \FOR {$k = 1$ to $k_{max}$}
    \STATE Update $\hat{\mathbf y}^{(k)}=\hat{P}\hat{\mathbf v}^{(k-1)}$, $\hat{\lambda}^{(k)}=(\hat{\mathbf v}^{(k-1)})^*\hat{\mathbf y}^{(k)}$, and $\hat{\mathbf v}^{(k)}=\frac{\hat{\mathbf y}^{( k)}}{\|\hat{\mathbf y}^{(k) }\|_2}$.
    \IF{$\|\hat{\mathbf y}^{(k) }-\hat{\mathbf v}^{(k-1) }\hat{\lambda}^{(k)}\|_{2^R}\le \gamma$}
    \STATE Compute $\hat{\mathbf w}^{(k-2)}$ and $\hat{\kappa}^{(k-2)}$ by \eqref{equ1} and \eqref{equ2}, respectively.
     \STATE If $\| \hat{P}\hat{\mathbf w}^{(k-2)}-\hat{\kappa}^{(k-2)}\hat{\mathbf w}^{(k-2)}\|_{2^R}\le \delta$, then Stop.
    \ENDIF
    \ENDFOR
    \STATE Compute $\hat{\mathbf u}^*=\mathcal{F}^{-1}(\hat{\mathbf w}^{(k-2)})$.
    \end{algorithmic}
\end{algorithm}

\section{A Novel Method for Computing All Eigenpairs}\label{section4}
In the previous section, we introduced an improved power method for computing the strictly dominant eigenvalues of dual quaternion Hermitian matrices and further enhanced its convergence rate by integrating Aitken extrapolation. However, the power method faces a limitation when the matrix possesses two eigenvalues with identical standard parts but distinct dual parts, rendering the method ineffective.  To address this challenge, in this section, we propose a novel algorithm for computing all eigenpairs of dual quaternion Hermitian matrices based on the relationship between dual complex adjoint matrices and dual quaternion matrices. To this end, we first present a method for computing the eigen-decomposition of dual complex Hermitian matrices, as described below.

\begin{theorem}\label{theorem3.3}
Let $\hat{P}= P_1+ P_2\varepsilon\in \mathbb{DC}^{n\times n}$ be a dual complex Hermitian matrix, $U$ be a unitary matrix such that $U^* P_1 U=\operatorname{diag} \left (\lambda _1I_{n_1},\lambda _2I_{n_2},\cdots, \lambda _t I_{n_t}\right )$, where $\lambda_1>\lambda_2>\cdots>\lambda _t$ and $n_1+n_2+\cdots+n_t=n$. Denote $U^* P_2 U=(P_{ij})$, where $P_{ij}\in C^{n_i\times n_j}$.  Let $U_i$ be a unitary matrix with dimension $n_i$ such that $U_i^* P_{ii}U_i=\operatorname{diag}\left (\mu^i _1,\mu^i _2,\cdots, \mu^i _{n_i}\right )$. Denote $V=\operatorname{diag}\left (U_1,U_2,\cdots, U_t\right )$ and $Q=V^*U^*P_2 U V=(Q_{ij})$, where $Q_{ij}\in C^{n_i\times n_j}$. Denote $T=(T_{ij})$, where $T_{ij}\in C^{n_i\times n_j}$, $T_{ii}=O_{n_i\times n_i}$ for $i=1,2,\ldots,n$, $T_{ij}=Q_{ij}/(\lambda_j -\lambda_i)$ for $i\ne j$. Let $\hat{U}=U V(I+T\varepsilon )$, then  $\hat{U}^*\hat{P}\hat{U}$ is a diagonal dual matrix, i.e., $$\hat{U}^*\hat{ P}\hat{U}=\operatorname{diag}\left ( \lambda_1+\mu^1_1\varepsilon , \cdots,\lambda_1+\mu^1_{n_1}\varepsilon ,\cdots, \lambda_t+\mu^t_1\varepsilon , \cdots,\lambda_t+\mu^t_{n_t}\varepsilon \right ).$$
In addition, $\hat{U}$ is a unitary dual complex matrix. Denote $\hat{\Sigma}=\hat{U}^*\hat{P}\hat{U}$, then
$\hat{U}\hat{\Sigma}\hat{U}^*$ is the eigendecomposition of $\hat{P}$.
\end{theorem}
\begin{proof}
Clearly, it holds
\begin{equation*}
\hat{ U}^*\hat{ P}\hat{ U}=  V^* U^* P_1 U V+( V^* U^* P_2 U V+ T^* V^* U^* P_1 U V+ V^* U^* P_1 U V T)\varepsilon.
\end{equation*}
Then the standard part of $\hat{ U}^*\hat{ P}\hat{ U}$ is
\begin{align*}
 V^* U^* P_1 U V
&= V^*\operatorname{diag}\left ( \lambda _1 I_{n_1},\lambda _2 I_{n_2},\cdots, \lambda _t I_{n_t} \right ) V
\\&=\operatorname{diag}\left (\lambda _1  U^*_1 U_1,\lambda
_2  U^*_2 U_2,\cdots,\lambda _t  U^*_t U_t\right )
\\&=\operatorname{diag}\left ( \lambda _1 I_{n_1},\lambda _2 I_{n_2},\cdots, \lambda _t I_{n_t} \right ) ,
\end{align*}
and the dual part of $\hat{ U}^*\hat{ P}\hat{ U}$ is
\begin{align*}
 V^* U^* P_2 U V+ T^* V^* U^* P_1 U V+ V^* U^* P_1 U V T
=(  Q_{ij}+\lambda _j T^*_{ji}+\lambda _i T_{ij}).
\end{align*}
For $i=j$, we have
$$ Q_{ii}+\lambda _i T^*_{ii}+\lambda _i T_{ii}= Q_{ii}= U_i^* P_{ii} U_i=\operatorname{diag}\left ( \mu^i _1,\mu^i _2,\cdots, \mu^i _{n_i}\right ).$$
For $i\ne j$, it follows from $ Q_{ij}= Q^*_{ji}$ that
\begin{align*}
 Q_{ij}+\lambda _j T^*_{ji}+\lambda _i T_{ij}
&= Q_{ij}+\lambda _j Q^*_{ji}/\left ( \lambda _i -\lambda _j  \right ) +\lambda _i Q_{ij}/\left ( \lambda _j -\lambda _i  \right )
\\&= Q_{ij}+\lambda _j Q_{ij}/\left ( \lambda _i -\lambda _j  \right ) +\lambda _i Q_{ij}/\left ( \lambda _j -\lambda _i  \right ) =O.
\end{align*}
This implies that
$$\hat{ U}^*\hat{ P}\hat{ U}=\operatorname{diag}\left ( \lambda_1+\mu^1_1\varepsilon , \cdots,\lambda_1+\mu^1_{n_1}\varepsilon ,\cdots, \lambda_t+\mu^t_1\varepsilon , \cdots,\lambda_t+\mu^t_{n_t}\varepsilon  \right ).$$ Since $ T_{ii}=O$, and for $i\ne j$, it holds that $$ T_{ij}+ T^*_{ji}= Q_{ij}/( \lambda _j -\lambda _i) + Q^*_{ji}/( \lambda _i -\lambda _j)=O,$$ we have $$\hat{ U}^*\hat{ U}=( I+  T^*\varepsilon) V^* U^* U V( I+ T\varepsilon)= I+( T+ T^*)\varepsilon=\hat{I}.$$
This implies that $\hat{U}$ is a unitary  dual complex matrix.
\qed\end{proof}

Corollary \ref{corollary3.1} shows that each eigenvector of $\hat{\mathbf Q}\in\hat{\mathbb{H}}^n$ corresponds to two orthogonal eigenvectors of its dual complex adjoint matrix $\hat{P}=\mathcal{F}(\hat{\mathbf Q})$. Let T denote an eigenvector set of $\hat{P}$. Then $\mathcal{F}^{-1}(T)$ forms an eigenvector set of $\hat{\mathbf Q}$. However, to obtain an orthogonal eigenvector set of $\hat{\mathbf Q}$ from $\mathcal{F}^{-1}(T)$, it may be necessary to remove redundant eigenvectors within $\mathcal{F}^{-1}(T)$. To this end, we propose Algorithm \ref{algorithm3.4.1}, which extracts an orthogonal set of eigenvectors corresponding to a given eigenvalue $\hat{\lambda}$ from a possibly redundant eigenvector set of $\hat{\mathbf{Q}}$.

\begin{algorithm}
    \caption{Dual Quaternion Eigenvectors Orthogonalization}\label{algorithm3.4.1}
    \begin{algorithmic}
    \REQUIRE An eigenvector set $\left \{\hat{\mathbf v}_i \right \}^{m}_{i=1}$ of $\hat{\mathbf Q}\in \hat{\mathbb H}^n$ with respect to eigenvalue $\hat{\lambda}$.
    \ENSURE $\left \{ \hat{\mathbf u}_k \right \}_{k=1}^{t}$.
    \STATE Let $t=1$ and compute $\hat{\mathbf u}_1=\frac{\hat{\mathbf v}_1}{\left \| \hat{\mathbf v}_1 \right \|_2}$.
    \FOR{j~=~2~:~m}
    \STATE Compute $\hat{\mathbf w}=\hat{\mathbf v}_j-\sum_{i=1}^{t}\hat{\mathbf u}_i\hat{\mathbf u}_i^*\hat{\mathbf  v}_j$. Denote $\|\hat{\mathbf w}\|_2=w_1+w_2\varepsilon$.
    \IF{$w_1>0$}
    \STATE Let $t=t+1$, compute $\hat{\mathbf u}_t=\frac{\hat{\mathbf w}}{\left \| \hat{\mathbf w} \right \|_2}$.
    \ENDIF
    \ENDFOR
    \end{algorithmic}
\end{algorithm}

We prove that Algorithm \ref{algorithm3.4.1} outputs an orthogonal eigenvector set of $\hat{\mathbf Q}$ with respect to eigenvalue $\hat{\lambda}$.

\begin{lemma}\label{lemma3.4}
Suppose that $\left \{ \hat{\mathbf v}_k \right \}_{k=1}^{m}\subseteq \hat{\mathbb{Q}}^{n\times 1}$ is an eigenvector set of $\hat{\mathbf Q}\in \hat{\mathbb H}^n$ with respect to eigenvalue $\hat{\lambda}$, then Algorithm \ref{algorithm3.4.1} outputs an orthogonal eigenvector set $\left \{\hat{\mathbf u}_i \right \}^{t}_{i=1}$ of $\hat{\mathbf Q}$ with respect to eigenvalue $\hat{\lambda}$.
\end{lemma}
\begin{proof}
We prove by induction. It is obvious when $t=1$.
Assume that when $t<s$, it holds $\hat{\mathbf u}^*_k\hat{\mathbf u}_l=\hat{0}$ for integers $k,l$ satisfying $1\le k,l\le t$ and $k\ne l$, and it holds $\hat{\mathbf Q}\hat{\mathbf u}_k=\hat{\mathbf u}_k\hat{\lambda}$ for $k=1,2,\ldots,t$.

When $t=s$, there exists integer $m(t)$ such that
$\hat{\mathbf u}_t=(\hat{\mathbf v}_{m(t)}-\sum_{i=1}^{t-1}
\hat{\mathbf u}_i\hat{\mathbf u}_i^*
\hat{\mathbf v}_{m(t)})/\hat{c}_t$, where $\hat{c}_t= \| \hat{\mathbf v}_{m(t)}-\sum_{i=1}^{t-1}
\hat{\mathbf u}_i\hat{\mathbf u}_i^*
\hat{\mathbf v}_{m(t)} \|_2\in \mathbb{D}$.
Therefore, based on induction assumption, for $k<s$, it holds
\begin{align*}
\hat{\mathbf u}^*_k\hat{\mathbf u}_t &=\frac{1}{\hat{c}_t}\hat{\mathbf u}^*_k\hat{\mathbf v}_{m(t)}-\frac{1}{\hat{c}_t}\sum_{i=1}^{t-1}\hat{\mathbf u}^*_k\hat{\mathbf u}_i\hat{\mathbf u}_i^*\hat{\mathbf v}_{m(t)}
\\&=\frac{1}{\hat{c}_t}\left(\hat{\mathbf u}^*_k\hat{\mathbf v}_{m(t)}-\hat{\mathbf u}^*_k
\hat{\mathbf u}_k\hat{\mathbf u}_k^*\hat{\mathbf v}_{m(t)}\right)
\\&=\frac{1}{\hat{c}_t}\left(\hat{\mathbf u}^*_k\hat{\mathbf v}_{m(t)}-\hat{\mathbf u}_k^*\hat{\mathbf v}_{m(t)}\right)
=\hat{0}.
\end{align*}
Similarly we have $\hat{\mathbf u}^*_t\hat{\mathbf u}_k=\hat{0}$.
Note that 
\begin{align*}
\hat{\mathbf Q}\hat{\mathbf u}_t&=\hat{\mathbf Q}\hat{\mathbf v}_{m(t)}\frac{1}{\hat{c}_t}-\hat{\mathbf Q}\sum_{i=1}^{t-1}\hat{\mathbf u}_i\hat{\mathbf u}_i^*
\hat{\mathbf v}_{m(t)}\frac{1}{\hat{c}_t}
\\&=\hat{\mathbf v}_{m(t)}\hat{\lambda}\frac{1}{\hat{c}_t}-\sum_{i=1}^{t-1}\hat{\mathbf u}_i\hat{\lambda}\hat{\mathbf u}_i^*\hat{\mathbf v}_{m(t)}\frac{1}{\hat{c}_t}
\\&=\frac{1}{\hat{c}_t}(\hat{\mathbf v}_{m(t)}-\sum_{i=1}^{t-1}\hat{\mathbf u}_i\hat{\mathbf u}_i^*\hat{\mathbf v}_{m(t)})\hat{\lambda}
=\hat{\mathbf u}_t\hat{\lambda},
\end{align*}
where the second to last equality holds because the eigenvalue $\hat{\lambda}$ is a dual number and it is commutative with dual quaternions. Therefore, $\hat{\mathbf u}_t$ is also an eigenvector of $\hat{\mathbf Q}$ with respect to the eigenvalue $\hat{\lambda}$.
\qed\end{proof}

By Corollary \ref{corollary3.1}, we can obtain all eigenpairs of a dual quaternion Hermitian matrix by computing all eigenpairs of the dual complex adjoint matrix. Based on Theorem \ref{theorem3.3} and Lemma \ref{lemma3.4}, we present a method for calculating all eigenpairs of dual quaternion Hermitian matrices (see Algorithm \ref{algorithm3.4}), with a proof provided in Theorem \ref{theorem3.4}.

\begin{algorithm}
    \caption{EDDCAM-EA} \label{algorithm3.4}
    \begin{algorithmic}
    \REQUIRE $\hat{\mathbf Q}\in \hat{\mathbb H}^n$.
    \ENSURE $\left \{\hat{\mathbf w}^i_j \right \}^{}_{i,j}$ and corresponding eigenvalues.
    \STATE Compute $\hat{P}=\mathcal{J}(\hat{\mathbf Q})$.
    \STATE Compute $\hat{U}$ by Theorem \ref{theorem3.3} such that $\hat{U}^*\hat{P}\hat{U}=\operatorname{diag}\left (\hat{\lambda}_1\hat{  I}_{t_1},\hat{\lambda}_2\hat{  I}_{t_2},\cdots,\hat{\lambda}_{l}\hat{ I}_{t_l}\right )$. Denote $\hat{U}=[ \hat{U}_1,\hat{U}_2,\cdots,\hat{U}_{l} ]$, where $\hat{U}_{i}= [ \hat{\mathbf u}^i_1,\hat{\mathbf u}^i_2,\cdots,\hat{\mathbf u}^i_{t_i}]\in \mathbb{DC}^{2n\times t_i}$.
    \STATE Compute $\hat{\mathbf v}^i_j=\mathcal{F}^{-1}(\hat{\mathbf u}^i_j),i=1,2,\ldots,l,j=1,2,\ldots,t_i$.
    \STATE Apply Algorithm \ref{algorithm3.4.1} to $\left \{\hat{\mathbf v}^i_j \right \}^{t_i}_{j=1}$ and get orthogonal eigenvectors $\left \{\hat{\mathbf w}^i_j \right \}^{m_i}_{j=1}$.
    \end{algorithmic}
\end{algorithm}

\begin{theorem}\label{theorem3.4}
Let $\hat{\mathbf Q}\in\hat{\mathbb{H}}^n$ and $\hat{P}=\mathcal{J}(\hat{\mathbf Q})$. Suppose that $\hat{U}$ is a unitary dual complex matrix such that
$\hat{U}^*\hat{P}\hat{U}=\operatorname{diag}\left (\hat{\lambda}_1\hat{I}_{t_1},\hat{\lambda}_2\hat{  I}_{t_2},\cdots,\hat{\lambda}_{l}\hat{ I}_{t_l}\right ).$
Denote $\hat{U}=[ \hat{U}_1,\hat{U}_2,\cdots,\hat{U}_{l}]$, where $\hat{U}_{i}= [ \hat{\mathbf u}^i_1,\hat{\mathbf u}^i_2,\cdots,\hat{\mathbf u}^i_{t_i}]\in \mathbb{DC}^{2n\times t_i}$.
Let $\hat{\mathbf v}^i_j=\mathcal{F}^{-1}(\hat{\mathbf u}^i_j),i=1,2,\ldots,l,j=1,2,\ldots,t_i$.
Then applying Algorithm \ref{algorithm3.4.1} to
$\left \{\hat{\mathbf v}^i_j \right \}^{t_i}_{j=1}$ will obtain a set of orthogonal eigenvectors $\left \{\hat{\mathbf w}^i_j \right \}^{m_i}_{j=1}$ of $\hat{\mathbf Q}$ with respect to eigenvalue $\hat{\lambda}_i$, and $\left \{\hat{\mathbf w}^i_j \right \}^{}_{i,j}$ is a complete eigenvector set of $\hat{\mathbf Q}$.
\end{theorem}
\begin{proof}
Theorem \ref{theorem3.3} provides a method for calculating the eigen-decomposition of dual complex Hermitian matrices. Consequently, $\hat{U}$ can be obtained using Theorem \ref{theorem3.3}. According to Corollary \ref{corollary3.1}, the eigenvalues of $\hat{P}$ always have even multiplicity, and each eigenvalue $\hat{\lambda}i$ of $\hat{\mathbf{Q}}$ has multiplicity $\frac{t_i}{2}$ for $i = 1, 2, \dots, l$. It follows from Corollary \ref{corollary3.1} and Lemma \ref{lemma3.4} that $\left \{\hat{\mathbf w}^i_j \right \}^{m_i}_{j=1}$ forms an orthogonal set of eigenvectors of $\hat{\mathbf Q}$ corresponding to the eigenvalue $\hat{\lambda}_i$. Thus, it remains to show that $m_i=\frac{t_i}{2}$.

We prove by contradiction that we assume that $m_i<\frac{t_i}{2}$.
Let $\tilde{\mathbf{w}}^i_j$ and $\tilde{\mathbf{v}}^i_j$ denote the standard parts of $\hat{\mathbf{w}}^i_j$ and $\hat{\mathbf{v}}^i_j$, respectively. 
According to Algorithm \ref{algorithm3.4.1},
the vector in $\left \{\tilde{\mathbf v}^i_j \right \}^{t_i}_{j=1}$ can be expressed as a linear combination of the vectors in $\left \{\tilde{\mathbf w}^i_j \right \}^{m_i}_{j=1}$. That is, there exist quaternions $\tilde{\alpha}_{jh},j=1,2,\ldots,t_i,h=1,2,\ldots,m_i$ such that $\tilde{\mathbf v}^i_j=\sum_{h=1}^{m_i}\tilde{\mathbf w}^i_h\tilde{\alpha}_{jh}$.
Denote $\tilde{\mathbf A}=(\tilde{\alpha}_{hj})$, $\tilde{\mathbf V}=[\tilde{\mathbf v}^i_1,\tilde{\mathbf v}^i_2,\cdots,\tilde{\mathbf v}^i_{t_i}]$ and $\tilde{\mathbf W}=[\tilde{\mathbf w}^i_1,\tilde{\mathbf w}^i_2,\cdots,\tilde{\mathbf w}^i_{m_i}]$, then
$\tilde{\mathbf V}=\tilde{\mathbf W}\tilde{\mathbf A}$.
It follows from (ii) in Lemma \ref{lemma3.2} that
$J(\tilde{\mathbf V})=J(\tilde{\mathbf W})J(\tilde{\mathbf A})$. Since $\left \{\tilde{\mathbf w}^i_j\right \}^{m_i}_{j=1}$ are orthogonal, then $\tilde{\mathbf W}^*\tilde{\mathbf W}=\tilde{\mathbf I}$. It follows from
(i), (ii) and (iv) in Lemma \ref{lemma3.2} that $J(\tilde{\mathbf W})^*J(\tilde{\mathbf W})=I_{2m_i}$, then the rank of complex matrix $J(\tilde{\mathbf W})$ is $2m_i$. Since $J(\tilde{\mathbf V})=J(\tilde{\mathbf W})J(\tilde{\mathbf A})$, then the rank of $J(\tilde{\mathbf V})$ is at most $2m_i<t_i$. Suppose that $\tilde{\mathbf V}= V_1+ V_2 \jj$, then $J(\tilde{\mathbf V})=\begin{bmatrix}
V_1  & V_2\\
-\overline{V_2}  &\overline{V_1}
\end{bmatrix}$.
Since the standard part of $\hat{U_i}$ is
$\begin{bmatrix}
V_1\\
-\overline{V_2}
\end{bmatrix}$, and the columns of $\hat{U_i}$ are orthogonal, then
$$\operatorname{rank}(J(\tilde{\mathbf V}))
=\operatorname{rank}\left(\begin{bmatrix}
V_1  & V_2\\
-\overline{V_2}  &\overline{V_1}
\end{bmatrix}\right)
\ge \operatorname{rank}\left(\begin{bmatrix}
V_1\\
-\overline{V_2}
\end{bmatrix}\right)=t_i.$$
Then the rank of $J(\tilde{\mathbf V})$ is at least $t_i$, which is a contradiction.
\qed\end{proof}

\section{Experimental Results}\label{section6}
This section presents the experimental results for the algorithms developed in the previous sections. All numerical experiments are conducted in MATLAB (2022a) on a laptop of 8G of memory and Inter Core i5 2.3Ghz CPU.

\subsection{Performance of Aitken Extrapolation}
We begin by evaluating the effect of Aitken extrapolation on the power method. A random dual quaternion Hermitian matrix $\hat{\mathbf{Q}}$ of dimension $n=20$ is generated, and the DCAM-PM and ADCAM-PM algorithms are employed to compute the strict dominant eigenvalue and corresponding eigenvector. To enable a direct comparison between the original iteration sequence and the sequence with Aitken extrapolation applied, the extrapolation is used from the outset. The results are presented in Fig. \ref{fig:testfig}. For reference, the strictly dominant eigenvalue $\lambda$ is precomputed with high accuracy using the DCAM-PM algorithm. Fig. \ref{fig:testfig} (a) shows the convergence of the eigenvalue, where $$|\lambda^{(n)}-\lambda|_{F^R}=\sqrt{|\lambda^{(n)}_{st}-\lambda_{st}|^2+|\lambda^{(n)}_{\mathcal I}-\lambda_{\mathcal I}|^2}.$$
Fig. \ref{fig:testfig} (b) shows the convergence of the error $e_{\lambda}$, where $e_\lambda=\|\hat{\mathbf{Q}}\hat{\mathbf{u}}-\lambda\hat{\mathbf{u}}\|_{2^R}$, $\hat{\mathbf{u}}$ is the eigenvector computed by DCAM-PM or ADCAM-PM.

The results in Fig. \ref{fig:testfig} demonstrate that Aitken Extrapolation effectively enhances the convergence rate of the power method.

\begin{figure}[!t]
    \centering
    \subfigure[]{
        \begin{minipage}[t]{0.48\linewidth}
            \centering
            \includegraphics[width=0.9\linewidth]{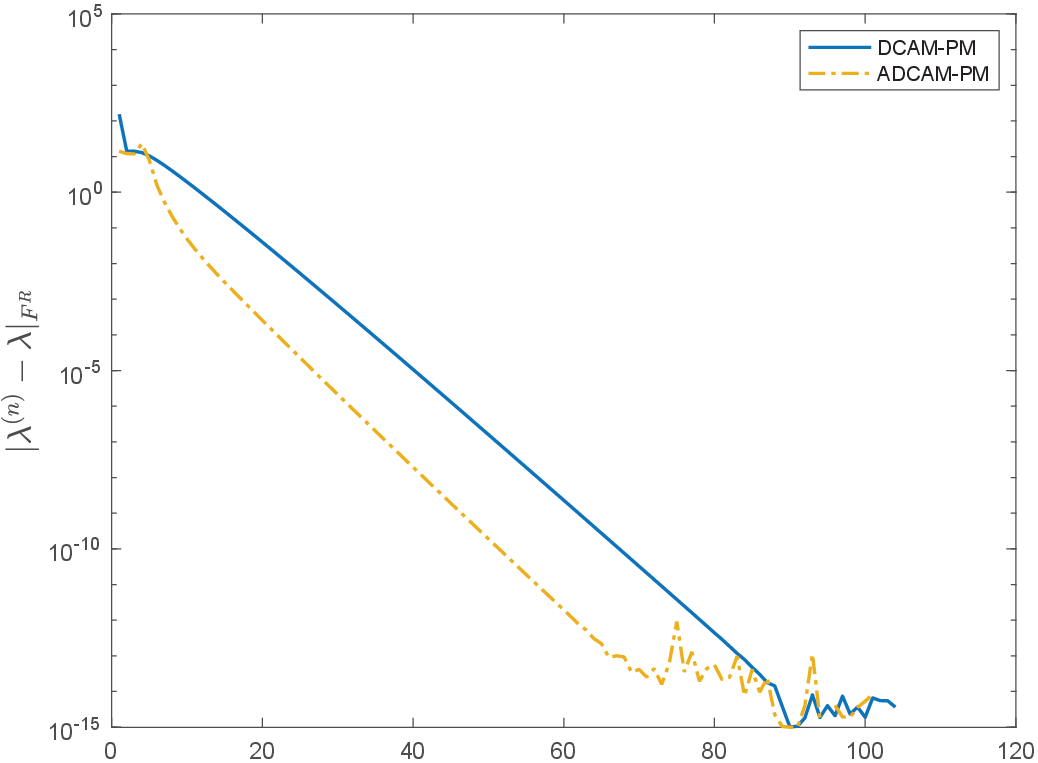}
        \end{minipage}
    }%
    \subfigure[]{
        \begin{minipage}[t]{0.48\linewidth}
            \centering
            \includegraphics[width=0.9\linewidth]{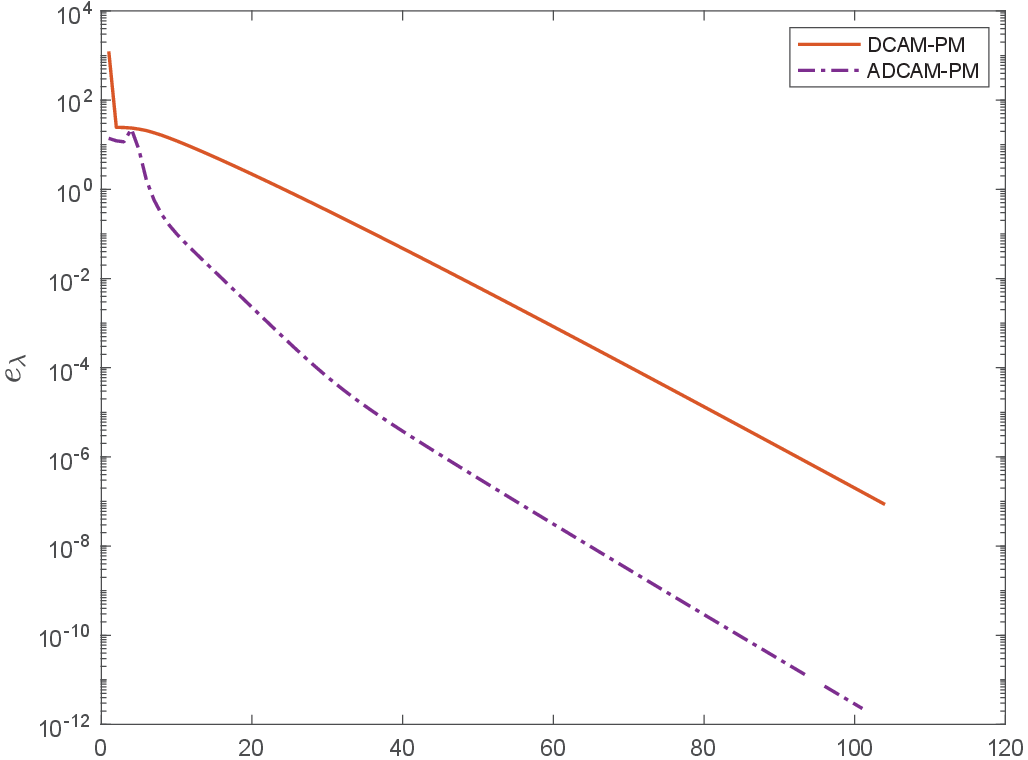}
        \end{minipage}
    }%
    \centering
    \caption{\small{Performance of Aitken extrapolation: (a) Convergence of eigenvalue;  (b) Convergence of error $e_{\lambda}$.}}
   \label{fig:testfig}
\end{figure}

Next, we compare the performance of the DCAM-PM and ADCAM-PM algorithms in solving the strict dominant eigenvalues of dual quaternion Hermitian matrices of varying dimensions. In the DCAM-PM algorithm, the parameter $\delta$ is set to $10^{-6}$, while in the ADCAM-PM algorithm, the parameter $\gamma$ is set to $10^{-3}$, and $\delta$ remains $10^{-6}$. The experimental results are summarized in Table \ref{table0},  where \textit{time(s)} denotes the average CPU time (in seconds) required to compute the strict dominant eigenvalues, and $n_{iter}$ represents the average number of iterations. All experiments were repeated 100 times, and the reported values are averages. The results indicate that Aitken extrapolation effectively reduces both the iteration count and computation time.

\begin{table}[htbp]
  \centering
  \caption{Performance of Aitken Extrapolation}
  \begin{tabular}{ccccccc}
    \hline
    \multirow{2}{*}{$n$} & \multicolumn{3}{c}{DCAM-PM} & \multicolumn{3}{c}{ADCAM-PM}\\
    \multicolumn{1}{c}{}& $e_\lambda$ & $n_{iter}$ & $time (s)$ & $e_\lambda$ & $n_{iter}$ & $time (s)$\\
    \hline
    10& 9.15e-7& 437.59 & 3.93e-3 & 2.35e-7 & 300.17 & 2.49e-3\\
    50& 8.24e-7& 48.06 & 1.77e-3&  5.32e-7& 33.35& 1.75e-3\\
    100& 7.78e-7& 38.94 & 3.37e-3& 4.01e-7& 25.59& 2.49e-3\\
    500& 7.51e-7&  31.80 & 7.88e-2& 2.53e-7& 20.95 & 5.55e-2\\
    1000& 7.21e-7& 31.65 & 3.50e-1& 1.56e-7& 21.18& 2.83e-1\\
    \hline
  \end{tabular}
  \label{table0}
\end{table}

\subsection{Application in Multi-Agent Formation Control}
The Laplacian matrix of the mutual visibility graph plays a pivotal role in multi-agent formation control \cite{Qi2023}. Notably, this Laplacian matrix is a dual quaternion Hermitian matrix, and computing its eigenvalues is essential for analyzing the dynamics of multi-agent formations. In this section, we evaluate the efficiency of the PM, DCAM-PM, and EDDCAM-EA algorithms by computing the eigenvalues of the Laplacian matrix.

In multi-agent formation control \cite{mac}, interactions between rigid bodies are modeled using an undirected graph $G = (V,E)$, where the vertex set $V$ represents the individual rigid bodies, and an edge $(i, j)\in E$ exists if and only if rigid bodies $i$ and $j$ have mutual visual perception. Consider the graph $G=\left ( V,E\right )$ in multi-agent formation control, where $|V|=n$ and $\hat{\mathbf q}=(\hat{q}_{i})\in \hat{\mathbb{U}}^{n\times 1}$ is a unit dual quaternion vector. The sparsity $s$ of graph $G$ is defined as $s=\frac{2|E|}{n^2}$.  The Laplacian matrix $\hat{\mathbf L}$ for the graph $G$, with respect to $\hat{\mathbf q}$, is given as defined in  \cite{Qi2023}
\begin{equation*}\label{L}
    \hat{\mathbf L}=\hat{\mathbf D}-\hat{\mathbf A},
\end{equation*}
where $\hat{\mathbf D}$ is a real diagonal matrix and its diagonal element is the degree of the corresponding vertex in graph $G$. Denote $\hat{\mathbf A}=(\hat{a}_{ij})$, then
\begin{equation*}
     \hat{a}_{ij}=
    \begin{cases}
    \hat{q}_{i}^{\ast } \hat{q}_{j}, & \text{if $\left ( i,j\right )\in E$},\\
    \hat{0}, & \text{otherwise}.
\end{cases}
\end{equation*}

To evaluate the efficiency of our algorithms, numerical experiments were performed on the graph $G$ with varying sparsity $s$ and dimensions $n=10$ and $n=100$. Let $e_\lambda=\frac{1}{n_0}\sum_{k=1}^{n_0}  \|\hat{\mathbf{L}}\hat{\mathbf{u}}_k-\hat{\lambda}_k\hat{\mathbf{u}}_k\|_{2^R}$, where $\{\hat{\lambda}_k\}_{k=1}^{n_0}$ and $\{\hat{\mathbf{u}}_k\}_{k=1}^{n_0}$ are eigenpairs of $\hat{\mathbf{L}}$ obtained by each algorithm, and $n_0$ is the total number of eigenvalues obtained by each algorithm. The numerical results are summarized in Table \ref{table1}. Here, $e_\lambda$ is used to verify the accuracy of the eigenvalues obtained by the algorithms, \textit{time(s)} represents the average CPU time (in seconds) required by each algorithm to compute eigenvalues. All experiments were repeated ten times with different selections of $\hat{q}$ and $E$.

\begin{table}[htbp]
  \centering
  \caption{Numerical results for eigenvalue problem in multi-agent formation control.}
  \begin{tabular}{cccccccc}
    \hline
    \multirow{2}{*}{$n$} & \multirow{2}{*}{$s$} & \multicolumn{2}{c}{PM}& \multicolumn{2}{c}{DCAMA-PM}&\multicolumn{2}{c}{EDDCAM-EA}\\
    \multicolumn{1}{c}{}&\multicolumn{1}{c}{}
    &$e_\lambda$ & $time (s)$ & $e_\lambda$ & $time (s)$& $e_\lambda$ & $time (s)$\\
    \hline
    10 & 10\% & 5.52e-11 & 1.92e-1  & 4.92e-11 & 6.32e-3 & 1.19e-13 & 1.42e-3\\
    10 & 20\% & 7.30e-11 & 4.87e-1  & 7.26e-11 & 7.75e-3& 2.95e-13 & 1.60e-3\\
    10 & 30\% & 7.75e-11 & 6.78e-1  & 7.84e-11 & 1.49e-2& 7.02e-13 & 1.48e-3\\
    10 & 40\% & 7.92e-11 & 1.20  & 8.01e-11 & 1.42e-2& 1.49e-12 & 1.89e-3\\
    10 & 50\% & 8.35e-11 & 1.40 & 8.11e-11 & 1.85e-2& 1.92e-12 & 1.89e-3  \\
    10 & 60\% & 8.39e-11 & 1.87  & 8.35e-11 & 1.99e-2& 1.37e-12 & 1.34e-3\\
    100 & 5\% & 9.75e-5 & 3.66e+1 & 9.74e-7 & 4.94& 9.73e-11 &  9.64e-2\\
    100 & 8\% & 9.88e-5 & 5.02e+1 & 9.78e-7 & 5.41& 2.12e-10 & 8.46e-2\\
    100 & 10\% & 9.98e-5 & 5.63e+1 & 9.80e-7 & 7.00 & 3.53e-10 &  9.92e-2  \\
    100 & 15\% & 1.10e-4 & 7.16e+1 & 9.82e-7 & 8.28 & 2.05e-10 &  9.13e-2\\
    100 & 18\% & 1.21e-4 & 7.94e+1 & 9.83e-7 & 9.69 & 1.85e-10 &  9.62e-2 \\
    100 & 20\% & 1.09e-4 & 9.00e+1 & 9.83e-7 & 9.42 & 4.51e-10 &  8.52e-2\\
    \hline
  \end{tabular}
  \label{table1}
\end{table}

Table \ref{table1} shows that for $n=10$, the DCAMA-PM algorithm achieves the same level of accuracy as the power method in terms of $e_\lambda$ while requiring less than $4\%$ of the computation time of the power method. For $n=100$, DCAMA-PM outperforms the power method in both accuracy and speed, using only $15\%$ of the computation time. These results demonstrate the effectiveness of dual complex adjoint matrices in assisting the power method in reducing computation time and enhancing accuracy. Notably, the EDDCAM-EA algorithm computes all eigenpairs in under 0.002 seconds for $n=10$ and under 0.1 seconds for $n=100$, with an average $e_{\lambda}<10^{-9}$. This performance significantly outpaces both the power method and DCAMA-PM, highlighting EDDCAM-EA's superior efficiency and accuracy in computing the eigenpairs of Laplacian matrices. 

\subsection{Addressing Power Method Limitations with the EDDCAM-EA Algorithm}\label{subsection6.3}
Numerical experiments indicate that the power method may fail to compute certain eigenvalues of a dual quaternion Hermitian matrix when multiple eigenvalues share the same standard part but differ in their dual parts. In this subsection, we demonstrate that the EDDCAM-EA algorithm effectively overcomes this limitation.

We generate a dual quaternion Hermitian matrix $\hat{\mathbf{P}}=(\hat{p}_{ij})$ as the following way. Let $\hat{\mathbf{q}}=(\hat{q}_i)$ be a random unit dual quaternion vector defined by
$$
\setlength{\arraycolsep}{1.2pt}
\left.\hat{\mathbf{q}}=\left[\begin{array}{rrrr}
   -0.5103 &  -0.2661 &  -0.2632 &  -0.7743\\
    0.2881 &  -0.6705 &  -0.2305 &  -0.6437\\
   -0.1236 &   0.1789 &  -0.7519 &  -0.6223\\
   -0.5605 &  -0.2485 &  -0.6001 &  -0.5138\\
   -0.5946 &  -0.1002 &  -0.2584 &  -0.7547 \end{array}\right.\right]+
\left[\begin{array}{rrrr}
   0.2645  &  -0.4286  &   0.4180  &  -0.1691\\
   -0.3885 &   -0.5378  &   0.2295 &   0.3042\\
   -0.9227 &   -0.9461  &   0.1770 &   -0.3027\\
   -0.2963 &   -0.3621 &   0.6937 & -0.3117\\
   -0.2488  &   0.2520  &   0.0635 &    0.1408
\end{array}\right]\varepsilon.$$
Let $E=\{\{ 1,2 \},\{ 2,3 \},\{ 3,4 \},\{ 4,5 \},\{ 5,1 \}\}$ and set $\hat{p}_{ij}$ as
\begin{equation*}
\hat{p}_{ij}=
\begin{cases}
 \hat{q}_{i}^{\ast } \hat{q}_{j}, & \text{if}~ ~\left \{ i,j \right \} \in E ,\\
i\varepsilon , & \text{if}~ ~ i=j ,\\
0, & \text{otherwise.} \\
\end{cases}
\end{equation*}

According to Theorem 9 in \cite{Cui2023}, the standard parts of the eigenvalues of $\hat{\mathbf{P}}$ are given by $cos(\frac{2j}{5}\pi)$ for $j=1,2,3,4,5$, implying that
$\hat{\mathbf{P}}$ possesses two eigenvalues with identical standard part.

Now, we compute all eigenvalues of $\hat{\mathbf{P}}$ by EDDCAM-EA. It terminates with $e_{\lambda }=3.0590\times 10^{-14}$ and the elapsed CPU time is $0.0118s$ and all five eigenvalues of $\hat{\mathbf{P}}$ are found as follows:
{\footnotesize
\begin{equation}\label{eig1}
2.0000+3.0000\varepsilon,\ 0.6180+3.5257\varepsilon,\ 0.6180+2.4743\varepsilon,\ -1.6180+3.8507\varepsilon,\ -1.6180+2.1493\varepsilon.
\end{equation}}

As shown in \eqref{eig1}, $\hat{\mathbf{P}}$ has two eigenvalues with identical standard parts but different dual parts. In such cases, the power method fails to compute all eigenvalues and does not converge when encountering eigenvalues that share the same standard part but differ in their dual parts. In contrast, the EDDCAM-EA algorithm accurately and efficiently computes all eigenvalues of $\hat{\mathbf{P}}$, demonstrating its clear advantage over the power method.

\section{Conclusions}\label{section7}

This paper investigates the eigenvalue computation problem for dual quaternion Hermitian matrices by transforming it into the computation of eigenpairs for dual complex matrices using the dual complex adjoint matrix. This transformation provides a theoretical foundation for the development of more efficient eigenvalue algorithms. Based on this framework, we introduce a novel variant of the power method, DCAM-PM, which significantly enhances computational efficiency. Furthermore, by integrating Aitken extrapolation, we accelerate convergence, leading to the ADCAM-PM algorithm, which achieves faster convergence than the original power method. To overcome the limitations of the power method, particularly when the matrix has eigenvalues with identical standard parts but different dual parts, we propose the EDDCAM-EA algorithm. EDDCAM-EA not only improves accuracy but also delivers substantial speed gains, effectively addressing scenarios where the power method becomes unreliable. Numerical experiments confirm the advantages of our proposed algorithms, demonstrating that Aitken extrapolation significantly reduces iteration counts and computation time. Our experimental results on eigenvalue computation for Laplacian matrices in multi-agent formation control further highlight the robustness and efficiency of DCAM-PM and EDDCAM-EA. These findings emphasize the potential of our methods for broader applications in computational mathematics and engineering.


%
%
\section*{Conflict of interest}
The authors declare that they have no conflict of interest.

\section*{Data availability}

The authors confirm that the data supporting the findings of this study are available
within the article.

%
\bibliographystyle{spmpsci}      

\begin{thebibliography}{99}
\bibitem{c1} Bultmann, Simon, Kailai Li, and Uwe D. Hanebeck. "Stereo visual SLAM based on unscented dual quaternion filtering." 2019 22th International Conference on Information Fusion (FUSION). IEEE, 2019.
\bibitem{qc1}Bryson, Mitch, and Salah Sukkarieh. "Building a robust implementation of bearing‐only inertial SLAM for a UAV." Journal of Field Robotics 24.1‐2 (2007): 113-143.
871): 381-395.

\bibitem{Cadena2016}Cadena, Cesar, et al. "Past, present, and future of simultaneous localization and mapping: Toward the robust-perception age." IEEE Transactions on robotics 32.6 (2016): 1309-1332.


\bibitem{Cui2023}
Cui, Chunfeng, and Liqun Qi. "A power method for computing the dominant eigenvalue of a dual quaternion Hermitian matrix." Journal of Scientific Computing 100.1 (2024): 21.

\bibitem{Daniilidis1999}
Daniilidis, Konstantinos. "Hand-eye calibration using dual quaternions." The International Journal of Robotics Research 18.3 (1999): 286-298.



\bibitem{Duan2023}Duan S Q, Wang Q W, Duan X F. On Rayleigh quotient iteration for dual quaternion Hermitian eigenvalue problem[J]. Mathematics, 2024, 12(24).


\bibitem{d2}Ling C, He H, Qi L. Singular values of dual quaternion matrices and their low-rank approximations[J]. Numerical Functional Analysis and Optimization, 2022, 43(12): 1423-1458.
 

\bibitem{d3}Ling C, Qi L, Yan H. Minimax principle for eigenvalues of dual quaternion Hermitian matrices and generalized inverses of dual quaternion matrices[J]. Numerical Functional Analysis and Optimization, 2023, 44(13): 1371-1394.


\bibitem{Li2023}Qi L, Luo Z. Eigenvalues and singular values of dual quaternion matrices[J]. Pacific Journal of Optimization, 2023, 19(2):257-272.

\bibitem{Qi2022}Qi L, Ling C, Yan H. Dual quaternions and dual quaternion vectors[J]. Communications on Applied Mathematics and Computation, 2022, 4(4): 1494-1508.




\bibitem{wang2023dual}Wang H, Cui C, Liu X. Dual r-rank decomposition and its applications[J]. Computational and Applied Mathematics, 2023, 42(8): 349.





\bibitem{Zhang1997}Zhang F. Quaternions and matrices of quaternions[J]. Linear algebra and its applications, 1997, 251: 21-57.
\bibitem{Grisetti2010}Grisetti, G., K¨ummerle, R., Stachniss, C., Burgard, W.: A tutorial on graph-based slam.
IEEE Intelligent Transportation Systems Magazine 2(4), 31–43 (2010)



\bibitem{ling2022neumann}Ling C, He H, Qi L, et al. von Neumann type trace inequality for dual quaternion matrices[J]. arXiv preprint arXiv:2204.09214, 2022.

\bibitem{chen2024dual} Chen Y, Zhang L. Dual Complex Adjoint Matrix and Applications to Hand-Eye Calibration and Multi-Agent Formation Control[J]. Journal of Scientific Computing, 2025, 104(1): 1-24.

\bibitem{Cheng2016}Cheng J, Kim J, Jiang Z, et al. Dual quaternion-based graphical SLAM[J]. Robotics and Autonomous Systems, 2016, 77: 15-24.
\bibitem{mac}Lin Z, Wang L, Han Z, et al. Distributed formation control of multi-agent systems using complex Laplacian[J]. IEEE Transactions on Automatic Control, 2014, 59(7): 1765-1777.

\bibitem{wei2024singular}Wei T, Ding W, Wei Y. Singular value decomposition of dual matrices and its application to traveling wave identification in the brain[J]. SIAM Journal on Matrix Analysis and Applications, 2024, 45(1): 634-660.


\bibitem{Wei2013}Wei E, Jin S, Zhang Q, et al. Autonomous navigation of Mars probe using X-ray pulsars: modeling and results[J]. Advances in Space Research, 2013, 51(5): 849-857.
\bibitem{Qi2023}Qi L, Wang X, Luo Z. Dual quaternion matrices in multi-agent formation control[J]. Communications in Mathematical Sciences, 2023, 21(7):1865-1874.

\end{thebibliography}

%
%

\end{document}